\documentclass[reqno,11pt]{amsart}
\usepackage{amsfonts}
\usepackage{amsmath}
\usepackage{amsthm}
\usepackage{amssymb}
\usepackage{mathtools}
\usepackage{mathrsfs} 
\usepackage{bbding}
\usepackage{bbm}
\usepackage{dsfont}
\usepackage{upgreek}
\usepackage[latin1,utf8]{inputenc}
\usepackage{libertine}
\DeclareMathAlphabet\EuScript{U}{eus}{m}{n}

\usepackage[numbers]{natbib}

\usepackage{hyperref}
\usepackage[capitalise]{cleveref}
\hypersetup{
    colorlinks=true,
    linkcolor=Violet,
    citecolor=Violet,
    urlcolor=Violet,
}

\DeclareFontFamily{U}{BOONDOX-calo}{\skewchar\font=45 }
\DeclareFontShape{U}{BOONDOX-calo}{m}{n}{
  <-> s*[1.05] BOONDOX-r-calo}{}
\DeclareFontShape{U}{BOONDOX-calo}{b}{n}{
  <-> s*[1.05] BOONDOX-b-calo}{}
\DeclareMathAlphabet{\mathlcal}{U}{BOONDOX-calo}{m}{n}
\SetMathAlphabet{\mathlcal}{bold}{U}{BOONDOX-calo}{b}{n}

\usepackage{fancyhdr}
\usepackage{enumerate}

\usepackage{cancel}
\usepackage{graphicx}
\usepackage{float}
\usepackage[dvipsnames]{xcolor}
\definecolor{cor_lucas}{rgb}{0,0,255}

\usepackage{accents}
\newcommand\thickbar[1]{\accentset{\rule{.4em}{.5pt}}{#1}}

\theoremstyle{plain}
\newtheorem{theorem}{Theorem}[section]
\newtheorem{lemma}[theorem]{Lemma}
\newtheorem{proposition}[theorem]{Proposition}
\newtheorem{proposition*}{Proposition}
\newtheorem{corollary}[theorem]{Corollary}
\newtheorem{claim}{Claim}[section]

\theoremstyle{definition}

\theoremstyle{remark}

\numberwithin{equation}{section}

\makeatletter
\newcommand{\ltext}[2]{%
  \@bsphack
  \csname phantomsection\endcsname % in case hyperref is used
  \def\@currentlabel{#1}{\label{#2}}%
  \@esphack
}
\makeatother

\DeclareMathOperator{\N}{\mathbb{N}}
\DeclareMathOperator{\Z}{\mathbb{Z}}
\DeclareMathOperator{\R}{\mathbb{R}}
\DeclareMathOperator*{\argmin}{arg\,min}
\DeclareMathOperator{\G}{\mathcal{G}}
\DeclareMathOperator{\Hh}{\mathcal{H}}
\DeclareMathOperator{\Om}{\Upomega}
\DeclareMathOperator{\F}{\mathscr{F}}
\DeclareMathOperator{\A}{\mathscr{A}}
\DeclareMathOperator{\p}{\mathbbm{P}}
\DeclareMathOperator{\E}{\mathbbm{E}}
\DeclareMathOperator{\K}{\mathtt{K}}
\DeclareMathOperator{\Kprime}{\mathtt{K}^\prime}

\title[FPP on RGGs: Asymptotic Properties]{Speed of Convergence and Moderate Deviations of FPP on Random Geometric Graphs}

\author[L. R. de Lima \and
D. Valesin]
{Lucas R. de Lima \ \and
Daniel Valesin}

\address{Department of Statistics\\
Institut of Mathematics and Statitics\\
University of S\~ao Paulo\\
Rua do Mat\~ao, 1010\\
05508-090 São Paulo - SP\\
Brazil}

\address{Center for Mathematics, Computation, and Cognition\\
Federal University of ABC\\
Av. dos Estados, 5001\\
09210-580 Santo Andr\'e - SP\\
Brazil}
\email{lrdelimath@gmail.com}
\email{lrdelima@ime.usp.br}

\address{Department of Statistics\\
University of Warwick\\
Coventry\\ 
CV4 7AL\\
United Kingdom}
\email{daniel.valesin@warwick.ac.uk}

\thanks{{\bf Funding:} Research supported by grants \#2019/19056-2, \#2020/12868-9, and \#2024/06021-4, S\~ao Paulo Research Foundation (FAPESP)}



\begin{document}

\begin{abstract}
    This study delves into first-passage percolation on random geometric graphs in the supercritical regime, where the graphs exhibit a unique infinite connected component. We investigate properties such as geodesic paths, moderate deviations, and fluctuations, aiming to establish a quantitative shape theorem. Furthermore, we examine fluctuations in geodesic paths and characterize the properties of spanning trees and their semi-infinite paths.
\end{abstract}

\maketitle

%\tableofcontents

\section{Introduction and main results}

In this paper, we extend the research initiated in~\cite{coletti2023} by conducting further analysis of first-passage percolation (FPP) on random geometric graphs (RGGs). Our primary focus is to examine the properties of geodesic paths and moderate deviations, which enhance our understanding of the convergence behaviour of first-passage percolation models towards their limiting shape---commonly referred to as the quantitative shape theorem. Additionally, we delve into the practical implications of these findings.

Kesten \cite{kesten1993} established the groundwork for understanding the speed of convergence in FPP models on the hypercubic lattice $\Z^d$ using martingales. More recently, Tessera \cite{tessera2018} refined these results through the application of Talagrand’s concentration inequality. Similar investigations have been conducted for FPP on random structures. For instance, Pimentel  \cite{pimentel2011} examined FPP on two-dimensional Delaunay graphs, Howard and Newman \cite{howard2001} studied Euclidean FPP, and Komj\'athy et al. \cite{komjathy2024} analyzed FPP on various random spatial networks with different growth dynamics.

This paper focuses specifically on FPP models on random geometric graphs in $\mathbb{R}^d$, defined as follows: Let~$d \in \N$ with~$d \ge 2$,~$\lambda \in (0,+\infty)$ and~$r \in (0,+\infty)$. Consider the random graph~$\G$ with vertex set~$V$ given as a homogeneous Poisson point process on~$\R^d$ with intensity~$\lambda$, and edge set $\mathcal{E} := {\{\{x,y\}: x,y \in V,\; \|x-y\| \le r\}}$, where~$\|\cdot\|$ denotes Euclidean distance. 

It is well known (see, for instance, Penrose~\cite{penrose1996}) that there exists a critical radius~$r_c(\lambda) \in (0,+\infty)$ such that if~$r < r_c(\lambda)$,~$\mathcal G$  has no infinite connected component almost surely, whereas when~$r > r_c(\lambda)$,~$\mathcal G$ contains a unique infinite connected component almost surely. Here, our focus lies on the latter scenario, known as the supercritical regime, and we denote the (almost surely unique) infinite connected component of~$\mathcal G$ by~$\mathcal H$. 

Let~$\tau$ be a non-negative random variable. Given a realization of the random graph~$\mathcal G = (V,\mathcal{E})$, we take random variables~$\{\tau_e: e \in \mathcal{E}\}$, all independent and distributed as~$\tau$. We define the first-passage time~$T(u,v)$ between vertices~$u$ and~$v$ of~$\mathcal H$ by~$T(u,v):= \inf_\gamma \sum_{e \in \gamma} \tau_e$, where the infimum is taken over all paths $\gamma$ from $u$ to $v$. We extend the function~$T(\cdot,\cdot)$ to pairs of vertices of~$\R^d$ as follows. Given~$x \in \R^d$, let~$q(x)$ be the vertex of~$\mathcal H$ that is closest to~$x$ in Euclidean distance (using some arbitrary, deterministic procedure to decide on a minimizer in case more than one exists). We then set~$T(x,y):= T(q(x),q(y))$ for~$x,y \in \R^d$. This gives rise to a random metric in~$\mathbb R^d$. We abbreviate~$T(x):=T(o,x)$, where~$o$ denotes the origin.

Let $\upupsilon_d$ denote the volume of the Euclidean ball of radius~$1$ in~$\R^d$. In~\cite{coletti2023}, this model was considered under the assumptions:
\begin{enumerate}
	\item[(${A}_1'$)] \ltext{${A}_1'$}{A1prime} $\p(\tau = 0) < (\upupsilon_d r^d\lambda)^{-1}$;
	\item[(${A}_2'$)] \ltext{${A}_2'$}{A2prime} $\E[\tau^\eta] < +\infty$ for some~$\eta > d+2$.
\end{enumerate}
Define
\begin{equation*}
H_t:=\{x \in \R^d: T(x) \le t\},\quad t \ge 0,
\end{equation*}
and for~$x \in \R^d$ and~$s \ge 0$, let~$B(x,s)$ denote the Euclidean ball with center~$x$ and radius~$s$.
One of the main results established there is that, assuming $r > r_c(\lambda)$ and that \eqref{A1prime} and \eqref{A2prime} hold, there exists a finite constant $\upvarphi > 0$ such that for any $\varepsilon > 0$, almost surely, for $t$ sufficiently large, we have
\[(1-\varepsilon) B(o,\upvarphi) \subseteq \frac{1}{t} H_t \subseteq (1+\varepsilon) B(o,\upvarphi).\]

Here, we establish a more informative version of the shape theorem under a more restrictive set of assumptions. Specifically, we assume that
\begin{itemize}
    \item[(${A}_1$)] \ltext{${A}_1$}{A1} $\p(\tau=0) =0$;
    \item[(${A}_2$)] \ltext{${A}_2$}{A2}
    $\E[e^{\eta \tau}] < +\infty$ for some $\eta>0$.
\end{itemize}
Our first main result is the following.
\begin{theorem}[Quantitative shape theorem for FPP on RGGs] \label{thm:speed.FPP}
	Let $d \geq 2$,~$\lambda > 0$ and $r>r_c(\lambda)$, and consider first-passage percolation on the random geometric graph on~$\R^d$ with parameters~$\lambda$ and~$r$, with passage times satisfying~\eqref{A1} and~\eqref{A2} above. 
Then, there exists $c'>0$ and a finite constant $\upvarphi>0$ such that, almost surely, for~$t$ large enough, we have
    \begin{equation} \label{eq:asymptotic.cone2}
	    \left(1-c'\frac{\log(t)}{\sqrt{t}}{}\right) B(o, \upvarphi) \subseteq \frac{1}{t}H_{t} \subseteq \left(1+c'\frac{\log(t)}{\sqrt{t}}{}\right) B(o, \upvarphi).
    \end{equation}
\end{theorem}

We derive this theorem as a consequence of two results of independent interest, which we now present.

\begin{theorem}[Moderate deviations of first-passage times] \label{thm_new_moderate_deviations}
	Consider first-passage percolation as in Theorem~\ref{thm:speed.FPP}, under the same assumptions as in that theorem. There exist~$C,C', c > 0$ such that for any~$x \in \mathbb R^d$ with~$\|x\|$ large enough, we have
	\begin{equation*}
		\p\left(\frac{|T(x) - \E[T(x)]|}{\sqrt{\|x\|}} > \ell \right) \le Ce^{-c \ell} \quad \text{for any } ~\ell \in \left[C'\log\big(\|x\|\big),\sqrt{\|x\|}\right].
	\end{equation*}
\end{theorem}

\begin{theorem}[Asymptotic expectation and variance of first-passage times] \label{thm:moderate.dev.FPP}
Consider first-passage percolation as in Theorem~\ref{thm:speed.FPP}, under the same assumptions as in that theorem. There exist finite constants~$\upvarphi > 0$ and~$C > 0$ such that, for~$x \in \R^d$ with~$\|x\|$ large enough,
	\begin{equation}\label{eq_asymp_exp}
		\frac{\|x\|}{ \upvarphi} \leq  \E\left[ T(x) \right]   \leq \frac{\|x\|}{ \upvarphi} + C \sqrt{\|x\|} \log\big(\|x\|\big)
	\end{equation}
	and
    \begin{equation} \label{eq:VarT}
    \operatorname{Var} T(x) \le C\|x\| \log \big(\|x\|\big).
    \end{equation}
\end{theorem}

The techniques employed in the proofs of the theorems above rely primarily on assumption \eqref{A2}, which enables a concentration inequality, allowing us to apply a renormalization procedure and extend the results of Garet and Marchand \cite{garet2010} and Kesten \cite{kesten1986,kesten1993} to this context. The main challenge is to control the spatial behaviour of the random structure and construct a suitable random process that approximates the FPP model on a RGG. It is worth noting that we find it reasonable to relax assumption \eqref{A1}; however, we restrict our analysis to the case where \eqref{A1} holds, as determining the sharpness of these conditions is beyond our scope.

The fact that the limiting shape is a Euclidean ball follows directly from the isotropic properties of the Poisson point process and the definition of RGGs. Some of the techniques and results in this article could potentially be extended to other types of random graphs. For instance, Komj\'athy at al.  \cite{komjathy2024} identified regimes with linear FPP growth in scale-free percolation, long-range percolation, and infinite geometric inhomogeneous random graphs. Once certain local spatial properties of these random graphs are verified, the approximation techniques we employ, together with linearity in the Euclidean distance, could lead to similar results in this class of random networks.

Returning to the case of FPP on RGGs, we aim at gaining deeper insights into the fluctuations of geodesic paths as a consequence of the moderated deviation of passage times. Our approach is influenced by the seminal work of \citet{howard2001}, from which we adapt their techniques to our model of FPP on random geometric graphs. This enables the investigation and quantification of fluctuations of the geodesic paths within our specific framework.

Given~$x,y \in \mathbb R^d$, a \emph{geodesic from~$x$ to~$y$} is a path~$\upgamma$ in~$\mathcal H$ from~$q(x)$ to~$q(y)$ that minimizes the passage time between these two vertices, that is, it satisfies~$\sum_{e \in \upgamma} \tau_e = T(x,y)$. Obviously, we can say ``geodesic from~$x$ to~$y$'' and ``geodesic from~$q(x)$ to~$q(y)$'' synonymously. 
Depending on the passage time distribution, there may be multiple geodesics between two points with positive probability.

To state our first theorem about geodesics, we introduce some more terminology and notation. Given~$U,V \subseteq \mathbb R^d$, the \emph{Hausdorff distance between~$U$ and~$V$} is
\[d_H(U,V) := \inf\big\{~\epsilon>0 ~\colon~ [U]_\epsilon \subseteq V \text{ and } [V]_\epsilon \subseteq U~\big\},\]
where for~$A \subseteq \R^d$ we write~$[A]_\epsilon := \bigcup_{u \in A}B(u, \epsilon)$, the $\epsilon$-neighbourhood of $A$.

For~$x,y \in \R^d$, we denote by~$\overline{xy}$ the straight line segment between $x$ and $y$. Given a path~$\upgamma=(q_0,\ldots,q_n)$ in~$\mathcal H$, we let~$\overline{\upgamma}$ be the polygonal path in~$\R^d$ constituted by the line segments $\overline{q_iq_{i-1}}$.

\begin{theorem}[Fluctuations of geodesics] \label{thm:fluctuations.of.geodesics}
    Let $d \geq 2$, $r>r_c(\lambda)$, and  $\epsilon\in (0,1/4)$. Consider  FPP with i.i.d.\ random variables
    defined on edges of the infinite component $\Hh$ of  $\G$; assume \eqref{A1} and \eqref{A2} are satisfied. Then, there exist $C,c>0$ such that, for all $x,y \in \R^d$,
    \begin{equation} \label{eq:geodesic.fluctuations}
        \p\left(\begin{array}{l}\text{for every geodesic $\upgamma$ from $x$ to $y$,}\\ \text{we have }d_H\big(\overline{\upgamma},\overline{xy}\big) \leq \|y-x\|^{\frac{3}{4} + \epsilon} \end{array}\right) \ge 1- C e^{-c \|y-x\|^{\epsilon}}.
    \end{equation}
\end{theorem}

\begin{figure}[!htb]
    \centering
    \includegraphics[scale=0.25]{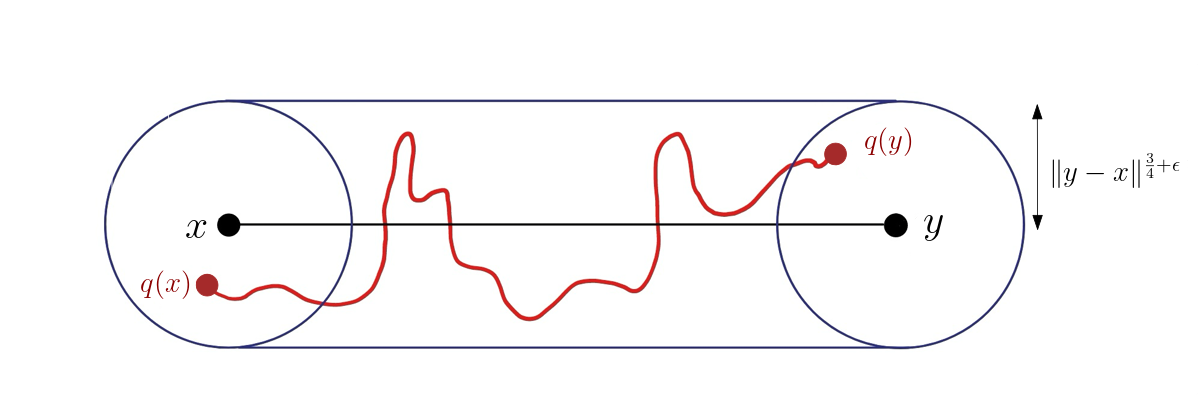}
    \caption{A geodesic path with its fluctuations in a cylindrical region of space given by \cref{thm:fluctuations.of.geodesics}.}
\end{figure}

For~$x,y \in \R^d$, let $\uptheta(x,y)$ denote the angle between $\vec{ox}$ and $\vec{oy}$, that is,~$\uptheta(x,y)=\arccos\left({\langle x,y\rangle}/{(\|x\|~\|y\|)}\right) \in [0,\pi)$ where $\langle \cdot,\cdot\rangle$ denotes the Euclidean inner product.

\begin{theorem}[Asymptotic behaviour of geodesics] \label{thm:f.straight.spanning.trees}
   Let $d \geq 2$, $r>r_c(\lambda)$, and~$\epsilon\in (0,1/4)$. Consider  FPP with i.i.d.\ random variables
    defined on edges of the infinite component $\Hh$ of  $\G$; assume \eqref{A1} and \eqref{A2} are satisfied. For any~$x \in \R^d$, there almost surely exists a finite set~$\mathscr F_x \subseteq \mathcal H$ such that the following holds. If~$u \in \mathcal H \backslash \mathscr F_x$ and~$\upgamma=(q_0,\ldots,q_n)$ is a geodesic started at~$q(x)$ with~$q_i = u$ for some ${i \in \{1,\ldots,n\}}$, then for all~$j \in \{i,\ldots, n\}$ we have~$\uptheta(u-x,q_j-x) \le \|u-x\|^{-\frac14+\epsilon}$.
\end{theorem}

It is illuminating to consider the above theorem in contexts where the geodesics are uniquely determined (which happens for instance if the passage time distribution is continuous). Then, the set of all geodesic paths started from some~$x \in \R^d$ form a tree rooted at~$q(x)$, which we denote by~$\mathlcal T_x$. Given another vertex~$v \in \mathcal H$, we let~$\mathlcal T_x^{\mathsf{out}}(v)$ denote the set of~$u \in \mathcal H$ such that the geodesic between~$x$ and~$u$ visits~$v$. The above theorem then says that, apart from finitely many choices of~$v \in \mathcal H$, the set~$\mathlcal T_x^{\mathsf out}(v)$ is contained in the cone ${\{y \in \mathbb R^d: \uptheta(v-x,y-x) \le \|v\|^{-\frac14+\epsilon}\}}$. In the terminology of~\cite{howard2001}, this says that~$\mathlcal T_x$ is \emph{$f$-straight}, where~$f(s):=s^{-\frac14+\epsilon}$.

\begin{figure}[!htb]
\label{fig_spanning}
    \centering
    \includegraphics[scale=0.32]{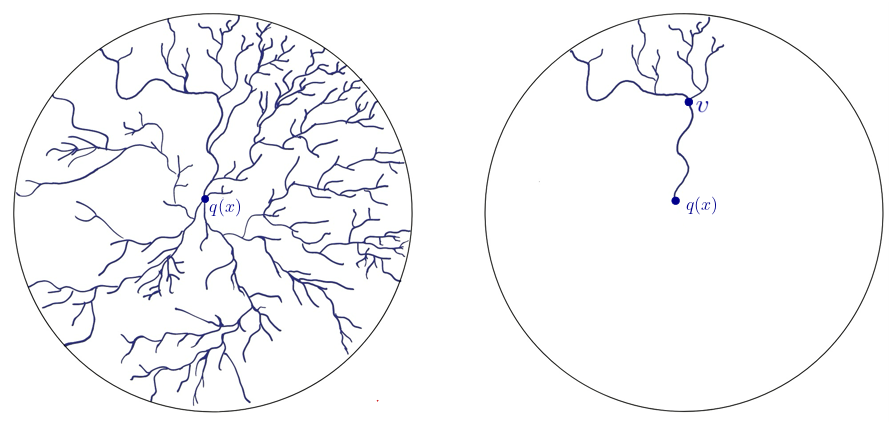}
    \caption{Section of a spanning tree $\mathlcal{T}_x$ (left) and the subtree given by $\mathlcal{T}_x^{\mathsf{out}}(v)$ (right).}
\end{figure}

\begin{figure}[!hbt]
    \centering
    \includegraphics[scale=0.1]{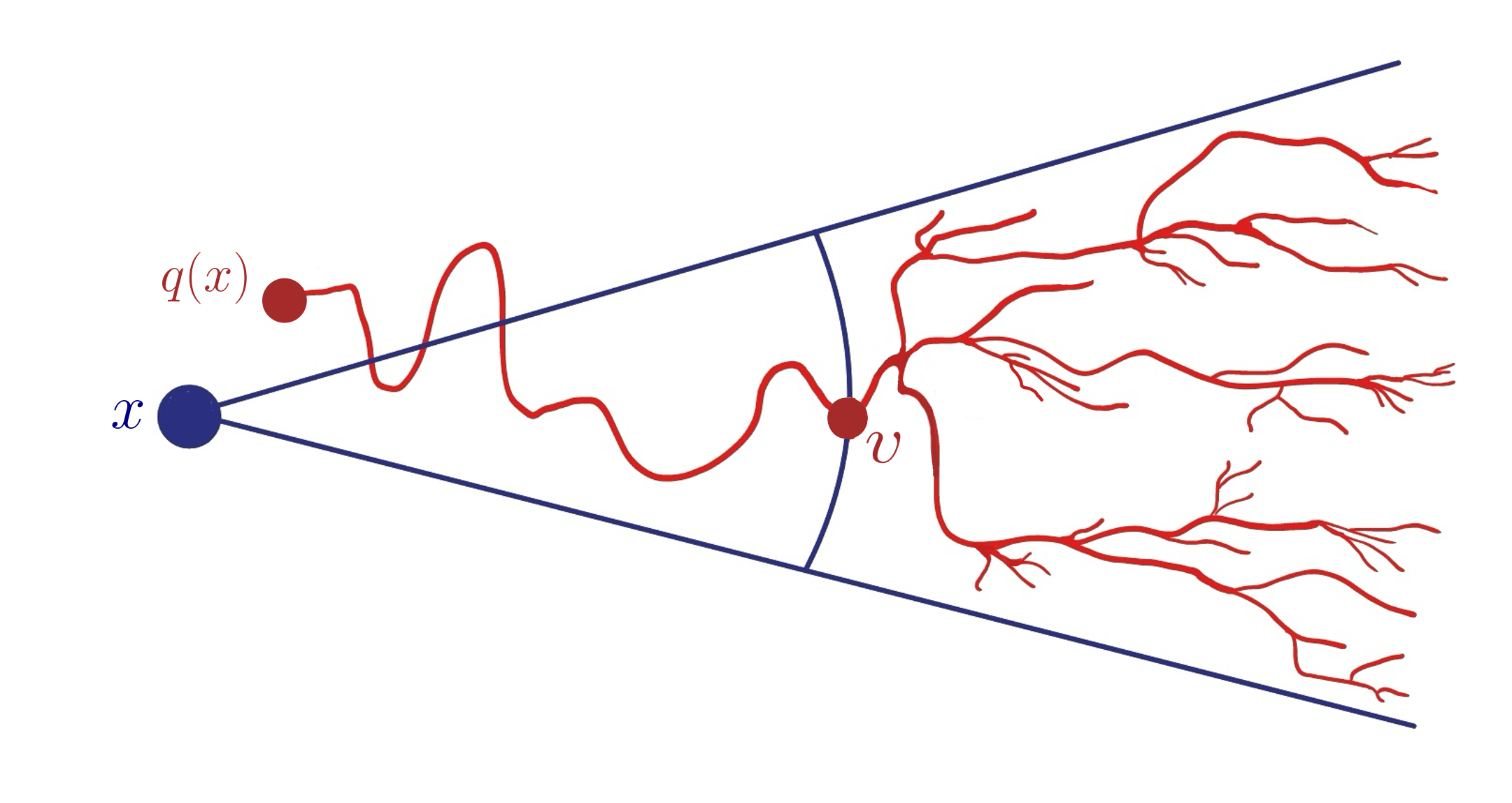}
    \caption{Geodesic paths crossing $v \in \Hh \setminus \mathscr{F}_x$ within the conic region described in \cref{thm:f.straight.spanning.trees}.}
\end{figure}

Now that we have established certain properties of geodesics, we will delve into the properties of the asymptotic behaviour of their semi-infinite paths. Let $\mathlcal{S}_x$ be the set of semi-infinite geodesic paths starting from $q(x)$, \textit{i.e.}, paths~$(q_0,q_1,\ldots)$ in~$\mathcal H$ such that~$q_0=q(x)$ and so that for each~$n$,~$(q_0,\ldots,q_n)$ is a geodesic.
For~$\upsigma=(q(x), q_1, q_2, \dots)$, we define the asymptotic direction of~$\upsigma$ as
\[\eth(\upsigma) := \lim_{n \uparrow +\infty} \frac{q_n}{\|q_n\|},\]
in case the limit exists. Note that, if it exists, we have~$\eth(\upsigma) \in\partial B(o,1)$.

\begin{corollary} \label{cor:asymp.dir}
    Consider $d \ge 2$ and $r> r_c(\lambda)$. Let the FPP be defined on $\Hh$ of $\mathcal{G}$ such that \eqref{A1} and \eqref{A2} hold true. Then, one has $\p$-a.s. that, for all $x \in \R^d$:
    \begin{enumerate}[(a)]
        \item Every semi-infinite path $\upsigma\in\mathlcal{S}_x$ has an asymptotic direction $\eth(\upsigma) \in \partial B(o,1)$.
        \item For all directions $\Hat{x} \in \partial B(o,1)$, there exist at least one semi-infinite path $\upsigma\in\mathlcal{S}_x$ with $\eth(\upsigma)=\Hat{x}$.
        \item The set of asymptotic directions with non-unique associated semi-infinite paths%$\mathlcal{T}_o$-path
        \[
            \mathlcal{D}_x =\big\{\Hat{x} \in \partial B(o,1)\colon\exists \upsigma_1,\upsigma_1 \in \mathlcal{S}_x \text{ with } \upsigma_1\neq\upsigma_2 \text{ s.t. } \eth(\upsigma_1)=\eth(\upsigma_2)=\Hat{x}\big\}
        \]
        is dense in $\partial B(o,1)$
    \end{enumerate}
\end{corollary}

\subsection{Notation.}

Throughout the text, we adopt the convention that $\|\cdot\|$ represents the Euclidean norm on $\R^d$. Additionally, $\|\cdot\|_1$ and $\|\cdot\|_\infty$ denote the $\ell^1$ and $\ell^\infty$ norms, respectively. The FPP model on RGGs are assumed to be defined on a probability space $(\Om, \A, \p)$. Here, $\N$ denotes the set of natural numbers $\{1, 2, \dots\}$, and $\N_0$ is defined as $\N\cup\{0\}$. We use $a\wedge b = \min\{a,b\}$ and $a \vee b = \max\{a,b\}$ for $a,b \in \R$, when it is convenient. All remaining notation is introduced within the text.

\subsection{Organization of the paper.}
This study is structured to prove the main results presented in the introduction. \cref{sec:basic.results} presents basic results on random geometric graphs and FPP are discussed, laying the groundwork for subsequent analysis. \cref{sec:approximation.scheme,sec:intermediate.results,sec:moderate.deviations,s_appendix_proofs} develop an approximation scheme for first-passage times and establish intermediate results crucial for understanding moderate deviations. Following this, expectation and variance bounds for first-passage times are established in \cref{sec:expectation.variance}, providing the probabilistic properties for the quantitative shape theorem, which is demonstrated in \cref{sec:quantitative.shape}. Fluctuations of geodesics and the properties of spanning trees are then explored in detail in \cref{sec:fluctuations.spanning.tree}, offering insights into their behaviour under the FPP model.

\section{Basic results on random geometric graphs and first-passage percolation}\label{sec:basic.results}
Consider the random set $\mathcal{P}_\lambda$, representing points in $\mathbb{R}^d$ generated by a homogeneous Poisson point process (PPP) with intensity $\lambda > 0$. Let us revisit the definitions provided in the introduction, by defining the \emph{random geometric graph (RGG)}~$\mathcal{G} = (V, \mathcal{E})$ on $\mathbb{R}^d$ as follows:
\[
  V = \mathcal{P}_\lambda \quad \text{and} \quad
  \mathcal{E} = \big\{\{u,v\} \subseteq V: \|u-v\|<r,~ u \neq v\big\}
\]
Since $\lambda^{-\frac{1}{d}}\mathcal{P}_\lambda \sim \mathcal{P}_1$, we consider $\lambda$ as a fixed parameter and allow $r$ to vary due to the homogeneity of the norm. Thus, unless specified otherwise, we set $\lambda=1$ and denote $\mathcal{P}_1$ as $\mathcal{P}$. Let $\left(\Upxi,\F,\mu\right)$ denote the probability space induced by the construction of $\mathcal{P}$. 

Our objective is to study the spread of an infection within an infinite connected component of $\mathcal{G}$. According to continuum percolation theory (see Penrose \cite[Chapter~10]{penrose1996}), for all $d \geq 2$, there exists a critical $r_c(\lambda)>0$ (or $r_c$ for $\lambda=1$) such that $\mathcal{G}$ has an (unique) infinite component $\Hh$ $\mu$-\textit{a.s.} for all $r>r_c$. Denoting the sets of vertices and edges of $\Hh$ as $V(\Hh)$ and $\mathcal{E}(\Hh)$ respectively, we often use $\Hh$ to represent $V(\Hh)$.

Let $\mathbbm{B}(t)$ denote the hypercube $[-t/2, t/2]^d$, and consider the Euclidean ball denoted as $B(x,t) := \{y \in \R^d \colon \|y-x\| < t\}$. Fix $\theta_{r} = \mu \big( B(o,r) \cap \Hh \neq \varnothing \big)$.  The following proposition presents a fundamental result concerning the volume of $\Hh$, which is a weaker version of Theorem 1 in Penrose and Pisztora~\cite{penrose1996}.

\begin{proposition}\label{prop:Hn.growth}
    Let $d \geq 2$, $r>r_c$ and $\varepsilon \in (0,1/2)$. Then, there exists $c_{0}>0$ and $t_0>0$ such that, for all $t \geq s_0$,
    \[
        \mu\left((1-\varepsilon)\theta_{r} <\frac{|\Hh \cap \mathbbm{B}(t) |}{t^d}  <(1+\varepsilon)\theta_{r} \right) \geq 1 -\exp(-c_{0}t^{d-1}).
    \]
\end{proposition}

Define $\operatorname{Cover}(t):= \mathbbm{B}(t) \cap \left( \bigcup_{v \in \Hh}B(v,r)\right)$ as the coverage of $\Hh$ in $\mathbbm{B}(t)$.  Let $v \in \mathbbm{B}(t)$. The Euclidean ball $B(v,t')$ is a \textit{spherical hole} when $ B(v,t') \cap \operatorname{Cover}(t) = \varnothing$. The largest diameter of a spherical hole in $\mathbbm{B}(t)$ is denoted by
\[
    \mathcal{D}(t) := \sup \left\{ t' \geq 0 \colon  \exists v \in \mathbbm{B}(t) \colon ~B(v,t'/2) ~\text{is a spherical hole in}~ \mathbbm{B}(t) \right\}.
\]

The following proposition is an adapted version of Theorem 3.3 of~\citet{yao2015}:

\begin{proposition} \label{prop:holes.H}
    Let $d \geq 2$ and $r>r_c$. Then, there exist $C,c>0$ such that, for all large enough $t$,
    \[\mu\big( c \cdot \log(t) < \mathcal{D}(t) < C \cdot\log(t) \big) \geq 1- \frac{1}{t^2}. \]
\end{proposition}

Let $u, v \in \R^d$, we define the \emph{Palm measure} $\mu_v$ as the convolution of the measure $\mu$ with the Dirac measure $\delta_v$. Specifically, $\mu_v := \mu \ast \delta_v$ and we introduce the notation $\mu_{u,v} := \mu \ast \updelta_u \ast \updelta_v$. Observe that $\p = \mu \otimes \nu$, where $\nu$ is a measure associated with the passage times (for further details, interested readers can refer to the construction of the probability space outlined in the proof of \cref{lem_apply_conc}). Similarly, $\p_{u,v}$ stands for
$\mu_{u,v} \otimes \nu$.

Define $\mathscr{P}(x,y)$ as the set of self-avoiding paths from $x$ to $y$ in $\G$. The simple length of a path $\gamma = (x=x_0, x_1, \dots, x_m=y) \in \mathscr{P}(x,y)$ is denoted by $|\gamma|=m$. Let $\thickbar{q}: \R^d \to \G$ be a function defined as follows:
\begin{equation} \label{def:q_bar.function}
    \thickbar{q}(x) := \argmin_{y \in V}\{\|y-x\|\}
\end{equation}
which determines the closest point to $x$ in $V$.
Note that from \eqref{def:q_bar.function}, $\thickbar{q}$ may have multiple values for certain $x \in \mathbb{R}^d$. In such cases, we presume that $\thickbar{q}(x)$ is uniquely defined by arbitrarily selecting one outcome of \eqref{def:q_bar.function}.

Let $\thickbar{D}(x,y)$ stand for the $\G$-distance between $x,y \in \R^d$ given by
\[
    \thickbar{D}(x,y) = \inf\{|\gamma|:\gamma \in \mathscr{P}\big(\thickbar{q}(x),\thickbar{q}(y)\big)\}.
\]

Set $\mathlcal{C}(x)$ to be the connected component of $\thickbar{q}(x)$ in $\G$. The subsequent result appears as Lemma 4 in \citet{yao2013} and as Lemma 3.4 in \citet{yao2011}.

\begin{lemma} \label{lm:PPP.clusters}
    Let $d \geq 2$ and $r>r_c$. Then, there exist $C,c>0$ such that, for each $t>0$,
    \begin{equation*}
    \mu_v\big(v \not\in \Hh , ~\mathlcal{C}(v) \not\subseteq B(v,t)\big) \leq C e^{-c t}.
    \end{equation*}
    Furthermore, there exist $C',c,>0$ and $\beta' >1$ such that, for all $u,v \in \R^d$ and every $t> \beta' \|u-v\|$,
    \begin{equation*}
        \mu_{u,v}\big(\thickbar{D}(u,v) \mathbbm{1}_{u \in\mathlcal{C}(v)}> t\big)  \leq C'e^{-c' t}.
    \end{equation*}
\end{lemma}

Let us proceed with defining the \emph{first-passage percolation} model on  $\G=(V,\mathcal{E})$ with i.i.d.\ \emph{passage times} $\{\tau_e: e \in \mathcal{E}\}$. We introduce $\thickbar{T}$ as a random variable called first-passage time such that, for all $x, y \in \R^d$,  we have
\[\thickbar{T}(x,y) := \inf\left\{\sum_{e \in \gamma} \tau_e \ \ \colon \ \ \gamma \in \mathscr{P}(\thickbar{q}(x),\thickbar{q}(y))\right\}.\]

The following lemma provides an upper bound for the tail distribution of regions within a connected component of $\G$.

\begin{lemma} \label{lm:T.bar.all}
    Let $d \geq 2$ and $r>r_c$, and suppose that \eqref{A2} holds true. Then, there exist constants $C,c>0$ and $\thickbar{\beta} >1$ such that, for all $u,v \in \R^d$ and every $t> \thickbar{\beta} \|u-v\|$,
    \[
        \p_{u,v}\big(\thickbar{T}(u,v) \mathbbm{1}_{u \in\mathlcal{C}(v)} \geq t\big)  \leq C e^{-c t}.
    \]
\end{lemma}
\begin{proof}
    Let~$\eta>0$ be a constant satisfying~\eqref{A2} and fix~$\beta''>0$ large enough that~$\mathbb E[e^{\eta \tau}] < e^{\eta \beta''}$. Then, set~$\bar{\beta}:=\beta' \beta''$, where~$\beta'$ is the constant given in Lemma~\ref{lm:PPP.clusters}.

    Fix~$u,v \in \mathbb R^d$ and~$t > \bar{\beta}\|u-v\|$. We bound
\begin{align*}
\nonumber    \p_{u,v}\big(u \in\mathlcal{C}(v), \;\thickbar{T}(u,v) > t \big)  \leq &\mu_{u,v}\big(u \in\mathlcal{C}(v), \;\thickbar{D}(u,v)  \ge t/\beta''\big)\\
    &+ \p_{u,v}\big(\thickbar{D}(u,v)  < t/\beta'',\; \thickbar{T}(u,v) > t \big).
\end{align*}
By Lemma~\ref{lm:PPP.clusters}, the first term on the right-hand side is smaller than~$C' e^{-c't/\beta''}$. To bound the second term, we first note that it is equal to
\begin{align}\label{eq_second_expr2}
    \E_{u,v}\big[ \mathds{1}_{\{ \thickbar{D}(u,v) < t/\beta''\}}\cdot \E_{u,v}\big[\thickbar{T}(u,v) \ge t \mid \mathcal G\big] \big]
\end{align}
by conditioning on the graph. Letting~$\G$ be a realization of the graph for which~$u \in \mathlcal{C}(v)$, let~$\gamma_{u\leftrightarrow v}(\G)$ be a self-avoiding path from~$u$ to~$v$ in~$\G$ whose number of steps is~$\thickbar{D}(u,v)$ (in case there are several such paths, we choose one according to some arbitrary procedure). 
Using a Chernoff bound, on the event~$\{ \thickbar{D}(u,v) < t/\beta''\}$, the conditional expectation in~\eqref{eq_second_expr2} is smaller than
\[ \frac{\E\left[\exp\{\eta \sum_{e\in\gamma_{u\leftrightarrow v}(\mathcal G)}\tau_e\}\right] }{e^{ \eta t}}  \leq \frac{\E[e^{\eta \tau}]^{\thickbar{D}(u,v)} }{e^{\eta t }} \leq \left(\frac{\E[e^{\eta \tau}] \vee 1}{e^{\eta \beta'' }}\right)^{t/\beta''}.\]
Integrating over~$\mathcal G$, the expectation in ~\eqref{eq_second_expr2} is then smaller than~$({(\E[e^{\eta \tau}]\vee 1)}/{e^{\eta \beta'' }})^{t/\beta''}$. The result now follows from the choice of~$\beta''$.
\end{proof}

Similarly to \eqref{def:q_bar.function}, we set $q: \R^d \to \Hh$ to be a function that determines the closest point to $x$ in $\Hh$ defined as
\begin{equation*}
    q(x) := \argmin_{y \in \Hh}\{\|y-x\|\}.
\end{equation*}

Define $D(u,v) := \thickbar{D}\big(q(u),q(v)\big)$ to represent the graph distance within the infinite connected component $\Hh$. Additionally, let
\[T(u,v) := \thickbar{T}(q(u),q(v))\] 
denote the \emph{first-passage time} on $\Hh$. For brevity, we write $T(x) = T(q(o),q(x))$.

The results below offer bounds for the probabilities concerning the growth and first-passage times within the infinite connected component $\mathcal{H}$.

\begin{lemma} \label{lm:T.bds}
    Let $d \geq 2$ and $r>r_c$. Then, there exist $C,c>0$ and $\beta^\dagger >1$ such that, for all $x \in \R^d$ and every $t> \beta^\dagger \|x\|$,
    \begin{equation} \label{eq:chem.all}
        \mu\big(D(o,x) \geq t\big)  \leq Ce^{-c t}. 
    \end{equation}

    Additionally, if \eqref{A2} holds, then there exist $c_{1}, c_{2}>0$ and $\beta>1$ such that, for all $x \in \R^d$ and every $t> \beta \|x\|$,
    \begin{equation} \label{eq:all.T}
        \p(T(x)\geq t) \leq c_{1} e^{-c_{2} t}, 
    \end{equation}
\end{lemma}
\begin{proof}
    Item \eqref{eq:chem.all} corresponds to Lemma 2.5 in \citet{coletti2023}. The proof of \eqref{eq:all.T} mirrors that of \cref{lm:T.bar.all}. Here, we replace $\thickbar{E}(t)$ with $E(t) := \{D(o,x) < t\}$ and $\thickbar{\beta}$ with $\beta := \beta^\dagger\beta''$. The result follows as a consequence of applying \eqref{eq:chem.all}.
\end{proof}

\begin{lemma} \label{lm:T.bds.sup}
    Let $d \geq 2$, $r>r_c$, and $\beta, \thickbar{\beta}>0$ be as defined in \cref{lm:T.bar.all,lm:T.bds}. If \eqref{A2} is satisfied, then there exist $c,c'>0$ such that, for all $t'>1$ and any ~$t>\beta t'$,
    \begin{equation} \label{eq:supT.ball}
        \p\left(\sup_{\|w\| < t'}T(w) > t\right) \leq c \exp(-c' t') + c (t')^d \exp(-c't).
    \end{equation}
    Moreover, there exist $C,C'>0$ such that, for all $t,t'>1$,
    \begin{equation}  \label{eq:supT.annulus}
        \p\left(\sup_{\substack{z \in \Hh\cap B(o,t')\\y \in \Hh\cap B(z,t)}} T(z,y) > \thickbar{\beta} t\right) \leq \exp(-C' t') + C (t't)^d \exp(-C't).
    \end{equation}
\end{lemma}
\begin{proof}
    Consider the constants random variables from \cref{prop:Hn.growth} to define the event
    \[E= \left\{ \|q(o)\| \le t'/2 \text{ and } |\Hh\cap B(o,2t')| < 2^d\theta_r\cdot(t')^d\right\}.\]
    Hence, by \cref{prop:Hn.growth} and \eqref{eq:all.T},
    \begin{align*}
        \p\left(\sup_{\|w\| < t'}T(w) > t\right) &\le \p(E^c) + \p\left( \left( \bigcup_{z \in \Hh \cap B(o,2t')} \hspace{-10pt}\left\{ T(z) > t\right\}\right) \cap E \right)  \\
        &\le 2 e^{-\frac{c_{0}}{4} t'} + 2^d\theta_rc_{1}(t')^d e^{-c_{2} t},
    \end{align*}
    which proves \eqref{eq:supT.ball}. Let us now define
    \begin{align*}
        E_1'&= \left\{ \theta_r (t')^d/2 < |\Hh\cap B(o,t')| < 3\theta_r\cdot(t')^d/2\right\}, \text{ and} \\
         E_2'&= \left\{ 
        \max_{z \in \Hh\cap B(o,t')} |\Hh\cap B(z,t)| \le \theta_r\cdot t^d \right\}.
        %E_2'&= \left\{ 
        %\text{for all } z \in \Hh\cap B(o,t'), \text{ we have }|\Hh\cap B(z,t)| \le \theta_r\cdot t^d \right\}.
    \end{align*}
    Observe that, by Mecke's formula, \cref{prop:Hn.growth,lm:T.bar.all}, there exist $C'',c''>0$ such that,
    \begin{align}
		\nonumber &\p\left(\sup_{\substack{z \in \Hh\cap B(o,t')\\y \in \Hh\cap B(z,t)}} T(z,y) > \thickbar{\beta} t\right) \\
		\nonumber &\le \p\big((E_1')^c\big) + \p\big((E_2')^c \cap E_1'\big) + \E \left[ \sum_{\substack{z \in \Hh\cap B(o,t') \\ y \in \Hh\cap B(z,t)}}\hspace{-10pt}\p_{z,y}\left(T(z,y)> \thickbar{\beta}t\right)\mathbbm{1}_{E_1'\cap E_2'} \right] \\
        \nonumber &\le e^{-\frac{c_{0}}{2}t'} + \theta_r(t')^d e^{-\frac{c_{0}}{2}t} + \theta_r^2(t't)^d C'' e^{-c''\thickbar{\beta} t}.
	\end{align}
        This inequality establishes \eqref{eq:supT.annulus} as asserted.
\end{proof}

\section{Approximation scheme for first-passage times} \label{sec:approximation.scheme}

To study the first-passage times on $\Hh$, we introduce a new random variable $T^t$ and a random graph $\G^t$, which builds upon $\G$ by incorporating additional vertices and edges. Subsequently, the first-passage times will be approximated by $T^t$ for a given $t>0$. We define $\G^t= (V^t, \mathcal{E}^t)$ with $t>0$ as follows:
\[
    V^t := V \cup (t\Z^d) \quad \text{and} \quad \mathcal{E}^t := \mathcal{E} \cup \mathcal{E}'(t).
\]
where
\[
    \mathcal{E}'(t):=\left\{ \{u,v\} \ \colon \ u \in t\Z^d \text{ and }\begin{array}{c}
    v \in \left(u + [-t/2, ~t/2)^d\right) \cap V\\ 
    \text { or }\\
    v \in t\Z^d \text { with } \|u-v\|=t
    \end{array}\right\}.
\]

The vertices corresponding to $t\Z^d$ are referred to as extra vertices, while $\mathcal{E}'(t)$ represents the set of extra edges. Since $V \cap (t\Z^d) = \varnothing$ with probability one, the unions determining $V^t$ and $\mathcal{E}^t$ are $\p$-a.s. disjoint. Analogous to $q(x)$ and $\thickbar{q}(x)$, we define, for all $x \in \R^d$,
~$q^t(x) := \argmin_{y \in V^t}\big\{ \| y-x \| \big\}$ and note that $q^t(o)=o$.

\begin{figure}[htb!]
    \centering
    \includegraphics[scale=0.3,trim={50 0 250 250},clip]{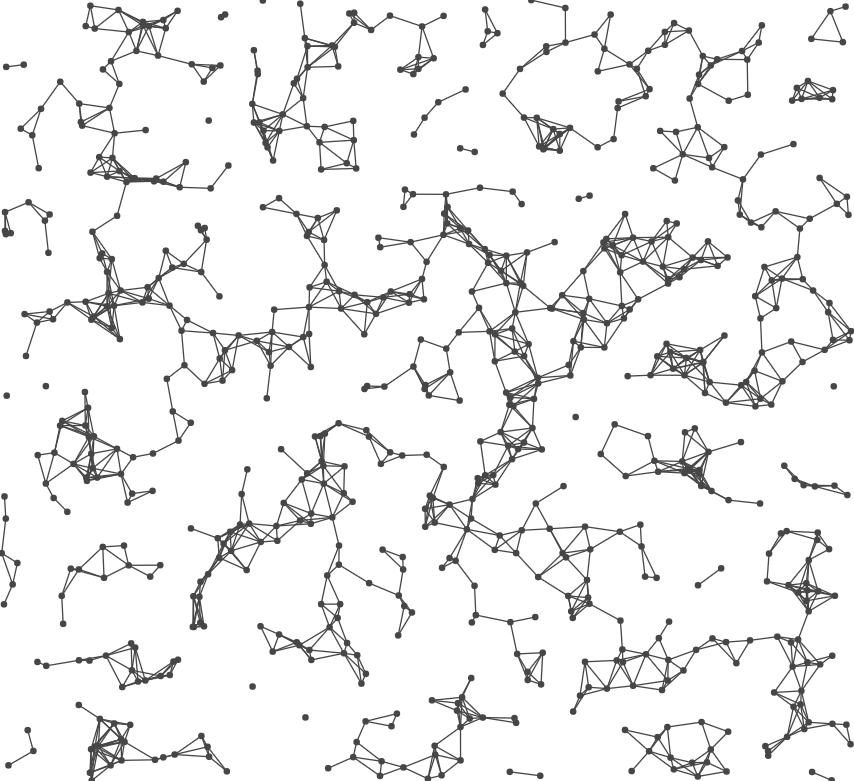} \hspace{0.4cm}\includegraphics[scale=0.3,trim={50 0 250 250},clip]{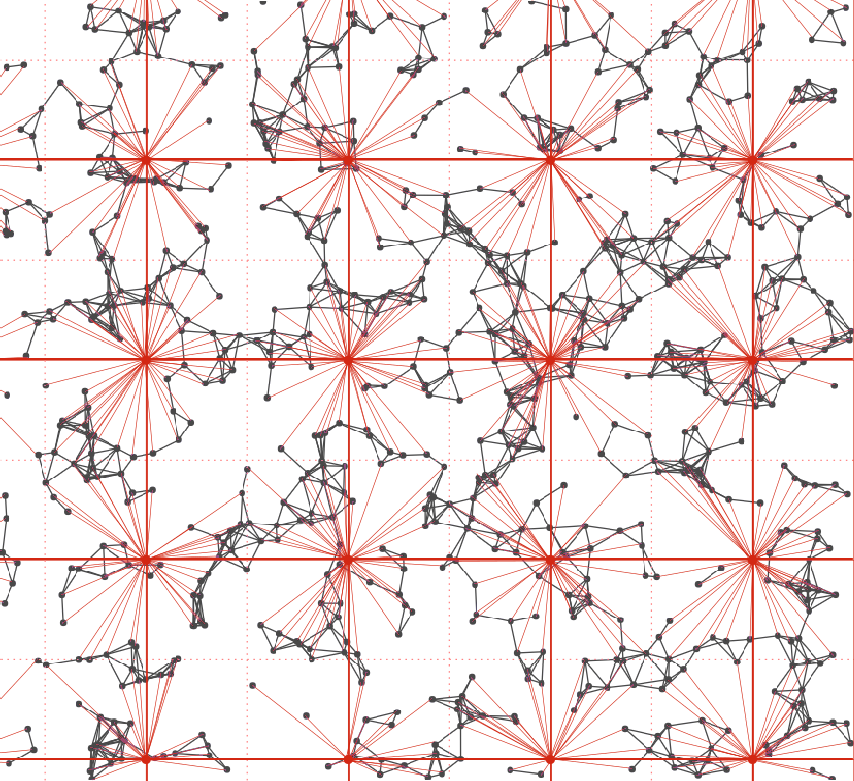}
    \caption{The image depicts the same region of a standard RGG, denoted as $\mathcal{G}$ (left), alongside an RGG with extra vertices and extra edges  $\mathcal{G}^t$ (right).}
    \label{fig:rgg_extra}
\end{figure}

The passage times along the extra edges are considered deterministically determined. Let $\K > 4 d(\thickbar{\beta} \vee \beta)$ be a fixed constant and note that $\K>1$ since $\beta,\thickbar{\beta}>1$. Define the passage times for the edges of $\G^t$ as
\[\tau_e^t := \left\{ \begin{array}{ll}
     \tau_e,& \text{ if }~e \in \mathcal{E},  \\ 
     \K t,& \text{ if }~e \in \mathcal{E}'(t).
\end{array}\right. \]

Let $\gamma$ be a (self-avoiding) path in $\mathcal G^t$. The passage time along $\gamma$ is given by
\[
    T^t(\gamma) := \sum_{e \in \gamma} \tau_e^t.
\]
We define the modified first-passage time~$T^t(x,y)$ between vertices~$x$ and~$y$ by~$T^t(x,y):= \inf_\gamma T^t(\gamma)$, where the infimum is taken over all paths $\gamma$ from $q^t(x)$ to $q^t(y)$ in $\G^t$. Henceforth, we assume that \eqref{A1} and \eqref{A2} hold true.

Now, we present the first results on the passage times on $\G^t$.

\begin{lemma}\label{lem_cheap_hop}
	There exists~$\delta > 0$  such that  for any~$n \in \N$, any~$\ell > 0$, and any~$t \ge 1$, 
	\begin{equation*}
		\p \left( \begin{array}{l} \text{there is 
 a self-avoiding path $\gamma$ in $\mathcal{G}^t$  starting in  } B(o,\ell)\\[.1cm] \text{starting in  }~B(o,\ell)~{ with }~|\gamma| = n \text{ and } \sum_{e \in \gamma} \tau^t_e \le \delta n
		\end{array}\right) \le  \frac{\max\{\ell^d,1\}}{2^{n}}.
	\end{equation*}
\end{lemma}
\begin{proof}
	Throughout this proof, we write, for~$k \in \N$ and~$s \in \R$,
	\[F_k(s):=\p(Z_1 + \cdots + Z_k \le s), \quad \text{where } Z_1,\ldots,Z_k \text{ are i.i.d. } \sim \tau,\]
	that is,~$F_k$ is the cumulative distribution function of the~$k$-fold convolution of the passage time through one (non-extra) edge.

	Fix~$n \in \N$. Let~$\varepsilon > 0$ be a small constant to be chosen later. Using the assumption that~$\lim_{s \to 0} F_1(s) = 0$,
    we can choose~$\delta > 0$ such that ${F_1(2\delta) < \varepsilon^2/4}$. We then bound, for all~$k \in \mathbb N$,
    \begin{equation}\begin{split}
        F_k(\delta k) &= \mathbb P\left(\sum_{i=1}^k Z_i \le \delta k\right) \le \mathbb P\left( \sum_{i=1}^k \mathds{1}_{\{Z_i \le 2\delta\}} \ge k/2\right) \\
&= \mathbb P\left(\mathrm{Bin}(k,F_1(2\delta)) \ge k/2\right) \le 2^k \cdot F_1(2\delta)^{k/2} < \varepsilon^k.
    \label{eq_choice_of_bar_delta}\end{split}
    \end{equation}
    We also assume that~$\delta < 1/2$.

	For now, we condition on a realization of the graph~$\mathcal G^t = (V^t, \mathcal{E}^t)$, so that the only randomness left is that of the passage times. Let~$\gamma$ be a (self-avoiding) path in~$\G^t$ with~$|\gamma| = n$ and let~$m$ denote the number of extra edges traversed by~$\gamma$. The probability that~$T^t(\gamma) \le \delta n$ is $F_{n-m}(\delta n - \K t m)$. This is zero in case~$m \ge \frac{\delta}{\K t} n$; otherwise, we bound:
			\[F_{n-m}(\delta n - \K t m) \le F_{n-m}(\delta(n - m)) \le \varepsilon^{n-m} \le \varepsilon^{n/2}, \]
			where the first inequality follows from~$\K t \ge \delta$, the second inequality follows from~\eqref{eq_choice_of_bar_delta}, and the third inequality follows from~$m < \frac{\delta}{\K t} n \le n/2$ (since~$\delta < 1/2$ and~$\K, t \ge 1$).

	This shows that in all cases, the probability that a path of graph length~$n$ has passage time (with respect to~$T^t$) smaller than~$\delta n$ is smaller than~$\varepsilon^{n/2}$.

	Now including also the randomness in the choice of the graph, a union bound over paths shows that the probability in the statement of the lemma is smaller than
			\begin{align} \label{eq_epsilon_with_E}
				\varepsilon^{n/2} \cdot \E  \left| \left\{ \text{self-avoiding paths $\gamma$ in $\mathcal{G}^t$  starting in  } B(o,\ell) \text{ with } |\gamma| = n				 \right\}\right|.
			\end{align}
			Defining
			\[\mathsf{v}_t(s) := \sup_{x \in \R^d} \E[|V^t \cap B(x,s)|],\quad s > 0,\]
			by Mecke's formula, the expression in~\eqref{eq_epsilon_with_E} is smaller than
	\begin{equation}\label{eq_epsilon_with_E2}
			\varepsilon^{n/2} \cdot \mathsf{v}_t(\ell)\cdot (\mathsf{v}_t(r))^n.
	\end{equation}
Recalling that~$\upupsilon_d$ denotes the volume of the unit ball in~$\mathbb R^d$, we bound, for any~$t \ge 1$ and~$s > 0$:
	\begin{align*}
		&\mathsf v_t(s) \le \upupsilon_d s^d + \lceil s/t\rceil^d \le \upupsilon_d s^d + \lceil s \rceil^d  \\
		&\hspace{2cm}\le \upupsilon_d s^d + \lceil s \rceil^d \cdot \mathbbm{1}_{\{s > 1\}} +  \mathbbm{1}_{\{s \le 1\}} \le (\upupsilon_d + 2^d) s^d + 1.
	\end{align*}
	Hence, the expression in~\eqref{eq_epsilon_with_E2} is smaller than
	\[\varepsilon^{n/2} \cdot ((\upupsilon_d + 2^d) \ell^d + 1) \cdot ((\upupsilon_d + 2^d)r^d + 1)^n.\]
	It is now easy to see that taking~$\varepsilon$ small enough (depending on~$r$ and~$d$, but not on~$\ell$ or~$n$), the right-hand side is smaller~$\max\{1,\ell^d\}/2^n$.
\end{proof}

For any~$u,v \in V^t$, let~$\gamma^t_{u \leftrightarrow v}$ denote the shortest path in~$\mathcal G^t$ from~$u$ to~$v$ that only uses extra edges.  Writing~$u=(u_1,\ldots,u_d)$ and~$v=(v_1,\ldots,v_d)$, we can bound
\begin{equation}
	\label{eq_hopcount_extra}
	|\gamma^t_{u \leftrightarrow v}| \le \sum_{i=1}^d \left\lceil \frac{u_i - v_i}{t}\right\rceil \le \frac{\|u-v\|_1}{t} + d \le  \frac{\sqrt{d}}{t} \|u-v\| + d.
\end{equation}
This gives
\begin{equation}
	\label{eq_T_t_and_l1}
		T^t(u,v) \le \K t |\gamma^t_{u\leftrightarrow v}| \le \K  \sqrt{d}\|u-v\| + \K t d.
\end{equation}

Concerning the special case of~$u = o$ and~$v = q^t(x)$ for~$x \in \R^d$, we will need the following.
\begin{claim}
	For any~$x \in \R^d$ with~$\|x\| \ge 1$ and any~$t \in [1,\|x\|]$, we have
	\begin{equation}\label{eq_newer_bound_path}
		|\gamma^t_{o \leftrightarrow q^t(x)}| \le \frac{3d}{t}\|x\|.
	\end{equation}
\end{claim}
\begin{proof}
We bound
\begin{equation}
	\label{eq_bound_qt}
	\|q^t(x)\| \le  \|x\| + \|q^t(x)-x\| \le \|x\| +\sqrt{d} \|q^t(x)-x\|_\infty \le \|x\|+ \tfrac{\sqrt{d}t}{2},
\end{equation}
	and combining this with~\eqref{eq_hopcount_extra} and the assumptions~$\|x\| \ge 1$ and~$t \in [1,\|x\|]$ gives
	\begin{align*}
		|\gamma^t_{o \leftrightarrow q^t(x)}| \le \frac{\sqrt{d}}{t} \|q^t(x)\| + d &\le  \frac{\sqrt{d}}{t} \|x\| + \frac{3d}{2} \\[.2cm]
		&\le \frac{d}{t}\|x\| + \frac{3d}{2}\cdot \frac{\|x\|}{t} \le \frac{3d}{t}\|x\|.
	\end{align*}
\end{proof}

Letting~$\delta$ be the constant given in Lemma~\ref{lem_cheap_hop}, we define
\begin{equation}
\label{eq_def_Kprime}
\Kprime := \frac{3d \K}{\delta},
\end{equation}
and
\begin{equation*}
	\begin{split}
		&Y_{t,x}:= \inf \left\{
			\sum_{e \in \gamma} \tau_e^t:\; \gamma \text{ is a path in $\mathcal{G}^t$ from $o$ to $q^t(x)$ with } |\gamma| \le  \Kprime \|x\| 
	\right\},\\
&\hspace{6.5cm} \text{for }x \in \R^d \text{ with } \|x\| \ge 1,\; t \in [1, \|x\|].
	\end{split}
\end{equation*}
Since~$\frac{3d}{t} \le \frac{3d \K}{\delta}$, it follows from the above claim that~$\gamma^t_{o \leftrightarrow q^t(x)}$ belongs to the set of paths whose infimum is taken in the definition of~$Y_{t,x}$. In particular,
	\begin{equation}
		\label{eq_bound_on_Y}
		Y_{t,x} \le \K t |\gamma^t_{o \leftrightarrow q^t(x)}| \stackrel{\eqref{eq_newer_bound_path}}{\le}  3d \K\|x\|.
	\end{equation}

We now compare the truncated passage time $Y_{t,x}$ with $T^t(o, q^t(x))$.

\begin{lemma}
	\label{lem_Y_and_T_t}
	If~$x \in \mathbb R^d$ with~$\|x\| \ge 1$ and~$t \in [1,\|x\|]$, then
	\[\p(Y_{t,x} \neq T^t(o,q^t(x))) \le  2^{-\Kprime\|x\|}.\]
\end{lemma}
\begin{proof}
	In the event~$\{Y_{t,x} \neq T^t(o,q^t(x))\}$, there exists a path~$\gamma$ in~$\mathcal G^t$ from~$o$ to~$q^t(x)$ which has~$|\gamma| > \Kprime\|x\|$ and
	\begin{align*}
		\sum_{e \in \gamma} \tau^t_e = T^t(o,q^t(x)) <  Y_{t,x} &\stackrel{\eqref{eq_bound_on_Y}}{\le} 3d\K \|x\|.
	\end{align*}
	We then have
	\[\frac{\sum_{e \in \gamma} \tau^t_e}{|\gamma|} \le \frac{3d \K\|x\|}{|\gamma|}<   \frac{3d \K\|x\|}{\Kprime \|x\|} = \delta.\]
	By Lemma~\ref{lem_cheap_hop} (with~$\ell = 1$), the existence of such a path has probability smaller than~$2^{-|\gamma|} <  2^{-\Kprime \|x\|}$.
\end{proof}

The following result is a comparison of the $T^t$-distance between vertices provided by $q$ and $q^t$.

\begin{lemma}
	\label{lem_q_and_no_q}
	There exist~$\mathfrak{C} \ge 1$, $\mathfrak{c} > 0$ such that, for any~$x \in \R^d$ and any~$t \ge 1$,
	\[\p(|T^t(o,q^t(x)) - T^t(q(o),q(x))|\ge  \mathfrak{C}t) < e^{-\mathfrak{c}t}.\]
\end{lemma}
\begin{proof}
Recall that~$q^t(o) = o$. For any~$s > 0$,  by  the triangle inequality, 
 \begin{align}
\nonumber     &\p(|T^t(o,q^t(x)) - T^t(q(o),q(x))|\ge  s)\\ &\le \p(T^t(o, q(o)) \ge s/2)+ \p(T^t(q^t(x), q(x)) \ge s/2). \label{eq_new_triangle}
 \end{align}
 Let us deal with the second term on the right-hand side. By~\eqref{eq_T_t_and_l1}, it is smaller than
	\begin{equation}
		\label{ex_norms2i}
		\p \left( \|q^t(x) - q(x)\|\ge \alpha \right).
	\end{equation}
    where we abbreviate~$\alpha:= (\tfrac{s}{2} - \K td)/(\K\sqrt{d})$. Since~$V \subseteq V^t$, we have~$\|q^t(x) - x\|\le \|q(x) - x\|$, so the triangle inequality gives~$\|q^t(x)-q(x)\| \le 2 \|q(x)-x\|$. 
	Using this, the probability in~\eqref{ex_norms2i} is at most
	\begin{equation}\label{eq_weird_s}
		\p \left(\|q(x)-x\| \ge \alpha/2 \right) = \p \left(\|q(o)\| \ge \alpha/2 \right).
	\end{equation}
A similar argument shows that the first term in~\eqref{eq_new_triangle} is bounded by the same value. We have thus proved that
\begin{equation}\label{eq_new_new_triangle}
    \p\big(|T^t(o,q^t(x)) - T^t(q(o),q(x))|\ge  s\big) \le 2 \p \left(\|q(o)\| \ge \alpha/2\right).
\end{equation}
 
	Now, using~\cref{prop:Hn.growth}, there exists~$\bar{s}$ (not depending on~$x$) such that, for any~$s' \ge \bar{s}$,
	\begin{equation}\label{eq_apply_pisz}\p \big(\|q(o)\| \ge s'\big) < \frac12 e^{-c s'}.\end{equation}
		By taking~$s=\mathfrak{C}t$ with~$\mathfrak{C}:=  2(d + 2\sqrt{d} \cdot  \bar{s}) \K $ and using the definition of~$\alpha$, the right-hand side of~\eqref{eq_new_new_triangle} equals~$2\p(\|q(o)\|\ge  \bar{s}\cdot t)$.
	Since~$t \ge 1$, we have~$\bar{s} \cdot t \ge \bar{s}$, so by~\eqref{eq_apply_pisz}, we can bound
	\[ \p(\|q(o)\| \ge \bar{s} \cdot  t) < \frac12 e^{-c \bar{s} t},\]
  so we set~$\mathfrak{c}:= c \bar{s}$ to complete the proof.
\end{proof}

The lemma below analyses the first-passage time $T$ and the modified random variable $T^t$. The proof is somewhat involved, and we postpone it to   Appendix~\ref{s_appendix_proofs}.

\begin{lemma}
	\label{lem_T_t_and_T}
	 	There exists~$\mathfrak{c}_1 > 0$ such that, for~$x \in \R^d$ with~$\|x\|$ large enough and~$t$ large enough (not depending on~$x$) with~$t \le \|x\|$, we have
	\[\p\big(T^t(q(o),q(x)\big) \neq T(x)) < \|x\|^{4d} e^{-\mathfrak{c}_1 t}.\]
\end{lemma}

We derive the following corollaries from the results above.

\begin{corollary}\label{cor_join_bounds}
	For~$x$ and~$t$ as in Lemma~\ref{lem_T_t_and_T}, we have
	\begin{equation}\label{eq_all_together_now}
		\p (| T(x) - Y_{t,x}| \ge \mathfrak C t ) \le 2^{-\Kprime \|x\|} + e^{-\mathfrak{c} t} + \|x\|^{4d} e^{-\mathfrak{c}_1 t},
	\end{equation}
	where~$\Kprime$ is defined in~\eqref{eq_def_Kprime},~$\mathfrak C$ and~$\mathfrak c$ are the constants of Lemma~\ref{lem_q_and_no_q}, and~$\mathfrak c_1$ is the constant of Lemma~\ref{lem_T_t_and_T}.
\end{corollary}
\begin{proof}
	This follows from putting together Lemma~\ref{lem_Y_and_T_t}, Lemma~\ref{lem_q_and_no_q} and Lemma~\ref{lem_T_t_and_T}.
\end{proof}

\begin{corollary}\label{cor_join_expectation}
	There exists~$\mathfrak C_1 > 0$ such that the following holds. Let~$x \in \R^d$ be such that~$\|x\|$ is large enough (as required in Lemma~\ref{lem_T_t_and_T}), and let~$t \in [\mathfrak C_1 \log\big(\|x\|\big),\|x\|]$.  Then,
	\begin{equation}
		\label{eq_for_second_mom}
	\E [ (T(x)- Y_{t,x})^2] \le 2(\mathfrak Ct)^2.
	\end{equation}
\end{corollary}

	\begin{proof}
Let~$\beta^* := \max\left( 3d\K,\; \beta \right)$,
	where~$\beta$ is the constant of~\cref{lm:T.bds}. We bound
	\begin{align}
		\nonumber &\E[(T(x)-Y_{t,x})^2] \\
		\nonumber &\le (\mathfrak Ct)^2 + \E[ (T(x) - Y_{t,x})^2 \cdot \mathbbm{1}{\{|T(x) - Y_{t,x}| \ge \mathfrak Ct \}}]\\
		\label{eq_with_fraks}&\le (\mathfrak Ct)^2 + \E[ (T(x) + \beta^* \|x\|)^2 \cdot \mathbbm{1}{\{|T(x) - Y_{t,x}| \ge \mathfrak Ct \}}],
	\end{align}
	where the second inequality follows from
	\[|T(x) - Y_{t,x}| \le T(x) + Y_{t,x} \stackrel{\eqref{eq_bound_on_Y}}{\le} T(x) + \beta^* \|x\|.\] We will now bound the expectation in~\eqref{eq_with_fraks} by considering the cases where~$T(x) \le \beta^* \|x\|$ and~$T(x) > \beta^* \|x\|$. First,
	\begin{equation*}\begin{split}
		&\E[ (T(x) + \beta^* \|x\|)^2 \cdot \mathbbm{1}{\{|T(x) - Y_{t,x}| \ge \mathfrak Ct, \; T(x) \le \beta^* \|x\| \}}]\\
		&\quad \le 4 (\beta^*)^2 \|x\|^2 \cdot \p(|T(x) - Y_{t,x}| \ge \mathfrak Ct) \\
		&\quad \stackrel{\eqref{eq_all_together_now}}{\le} 4 (\beta^*)^2 \|x\|^2 \cdot (2^{-\Kprime \|x\|} + e^{-\mathfrak{c} t} + \|x\|^{4d} e^{-\mathfrak{c}_1 t}).
	\end{split}
	\end{equation*}
Second,
	\begin{equation*}\begin{split}
		&\E[ (T(x) + \beta^* \|x\|)^2 \cdot \mathbbm{1}{\{|T(x) - Y_{t,x}| \ge \mathfrak Ct, \; T(x) > \beta^* \|x\| \}}]\\
		&\quad \le \E[ (2T(x))^2 \cdot \mathbbm{1} \{T(x) > \beta^* \|x\|\}]\\
		&\quad \le 4  \E[(T(x))^2]^{1/2} \cdot \p(T(x) > \beta^* \|x\|)^{1/2}\\
		&\quad \le 4  \E[(T(x))^2]^{1/2} \cdot \left(c_{1} \exp(-c_{2} \beta^* \|x\|) \right)^{1/2},
	\end{split}
	\end{equation*}
where the second inequality is Cauchy-Schwarz, and the third inequality is given in \cref{lm:T.bds}. Next, we bound
\begin{align*}
	\E[(T(x))^2] &=2 \int_0^\infty s \cdot  \p(T(x) > s) \;\mathrm{d}s\\
	&\le 2 \beta \|x\| + 2 \int_{\beta\|x\|}^\infty s \cdot c_{1} \exp(-c_{2} s) \;\mathrm{d}s.
\end{align*}

Putting things together, we have shown that~$\E[(T(x)- Y_{t,x})^2]$ is bounded from above by
\begin{align*}
	&(\mathfrak C t)^2 + 4 (\beta^*)^2 \|x\|^2 \cdot (2^{-\Kprime \|x\|} + e^{-\mathfrak{c} t} + \|x\|^{4d}e^{-\mathfrak{c}_1 t})\\
	& + 4 \left( 2 \beta \|x\| + 2 \int_{\beta\|x\|}^\infty s \cdot c_{1} \exp(-c_{2} s) \;\mathrm{d}s \right)^{1/2} \cdot \left(c_{1} \exp(-c_{2} \beta^* \|x\|) \right)^{1/2}.
\end{align*}
It is now easy to see that if~$\mathfrak C_1$ is large and~$t \ge \mathfrak C_1 \log \big(\|x\|\big)$ with~$\|x\|$ large enough, the expression above  is smaller than~$2(\mathfrak C t)^2$.
\end{proof}

We conclude this section by stating a concentration result involving the truncated passage times~$Y_{t,x}$. 

\begin{proposition}\label{prop_new_concentration}
    There exists~$C_{\mathrm{dev}}>0$ such that the following holds. Let~$x \in \R^d$ be such that~$\|x\|$ is large enough,~$t \in [1,\|x\|$ and~$s > 0$. Then,
\begin{equation}\label{eq_Markov_for_dev}
	\p\big(|Y_{t,x} - \E[Y_{t,x}]| > s \big) \le 2 \inf_{\alpha \in \left(0,~1/(4\K t)\right)} \exp\left\{ C_{\mathrm{dev}} \alpha^2 t~\|x\| - s \alpha\right\}.
\end{equation}
\end{proposition}
The proof of this proposition requires a slight digression into a concentration inequality taken from the literature. In order to keep the flow of the exposition, we postpone this proof to Section~\ref{sec:intermediate.results}.

\section{Moderate deviations of first-passage times} \label{sec:moderate.deviations}

In this section, we present the proof of our second main theorem regarding the moderate deviations of first-passage times, utilizing the preparatory results established earlier.

\begin{proof}[Proof of Theorem~\ref{thm_new_moderate_deviations}]
	Using Jensen's inequality and~\eqref{eq_for_second_mom}, for~$x$ with~$\|x\|$ large enough and any~$t \in [\mathfrak{C}_1 \log(\|x\|),\|x\|]$ we have
	\begin{align*}
		|\E[T(x)] - \E[Y_{t,x}]| \le \E[|T(x)-Y_{t,x}|] \le (\E[(T(x) - Y_{t,x})^2])^{1/2} \le \mathfrak{C}\sqrt{2} t.
	\end{align*}
	In particular, if~$\mathfrak{C}\sqrt{2}t \le s/2$, then~$|\E[T(x)] - \E[Y_{t,x}]|\le s/2$ and
	\begin{equation*}
		\p(|T(x) - \E[T(x)]| > s) \le \p(|T(x) - \E[Y_{t,x}]| > s/2).
	\end{equation*}

	Next, in case~$ \mathfrak{C} t \le s/4$ we can bound
	\begin{align*}
		&\p\big(|T(x) - \E[Y_{t,x}]| > s/2\big) \\[.2cm]
		&\le \p\big(|T(x) - Y_{t,x}| > \mathfrak{C} t\big) + \p\big(|Y_{t,x} - \E[Y_{t,x}]| > s/4\big)\\[.2cm]
		&\le 2^{-\Kprime \|x\|} + e^{-\mathfrak{c}t} + \|x\|^{4d} e^{-\mathfrak{c}_1 t} + 2 \inf_{\alpha \in (0,1/(4 \K t))} \exp\{C_{\mathrm{dev}} \alpha^2 t~\|x\| - s \alpha/4\},
	\end{align*}
	using~\eqref{eq_all_together_now}and \eqref{eq_Markov_for_dev}. Therefore, we have proved that
\begin{equation}\label{eq_dev_almost}
	\begin{split}
		&\p\big(|T(x) - \E[T(x)]| > s\big) \\[.2cm]
		&\le 2^{-\Kprime \|x\|} + e^{-\mathfrak{c}t} + \|x\|^{4d}e^{-\mathfrak{c}_1 t} + 2 \inf_{\alpha \in (0,1/(4 \K t))} \exp\{C_{\mathrm{dev}} \alpha^2 t~\|x\|- s \alpha/4\},
	\end{split}
\end{equation}
	under the restrictions
\begin{equation}
	\label{eq_to_optimize}
	\|x\| \text{ large},\quad s > 0,\quad \mathfrak C_1 \log\big(\|x\|\big) \le t \le \min \left\{ \|x\|, \frac{s}{4\mathfrak{C}}\right\}.
\end{equation}
    Let us now fix $t=s/\sqrt{\|x\|}$, then, for sufficiently large $\|x\|$,
    \begin{equation*}
	\mathfrak C_1 \log\big(\|x\|\big) \sqrt{\|x\|}\le s \le \|x\|.
    \end{equation*}
    We also set~$\alpha = \frac{1}{4\K C_{\mathrm{dev}} \sqrt{\|x\|}}$ and note that $\alpha<\tfrac{1}{4\K t}$. With these choices for~$t$ and~$\alpha$, the right-hand side of~\eqref{eq_dev_almost} becomes
\[2^{-\Kprime \|x\|} + e^{-\mathfrak{c}s/\sqrt{\|x\|}} + e^{-\mathfrak{c}_1 s/\sqrt{\|x\|}} + 2  \exp\left\{  - ~\frac{\K-1}{16 \K^2 C_{\mathrm{dev}} } \cdot \frac{s}{\sqrt{\|x\|}} \right\}.\]
Clearly, there exist~$C,c > 0$ such that the expression above is smaller than~$Ce^{-cs/\sqrt{\|x\|}}$, uniformly over~$s \in [ \mathfrak{C}_1 \log\big(\|x\|\big)\sqrt{\|x\|}, ~ \|x\|]$. 
\end{proof}

\section{Concentration for the approximation of first-passage times} \label{sec:intermediate.results}
The goal of this section is to prove Proposition~\ref{prop_new_concentration}. Define the collection of boxes
\[
\mathscr B:= \{ z+ [-t/2,t/2)^d:\; z \in t \Z^d\}.
\]

The following lemma offers a bound for the boxes in $\mathscr B$ crossed by the truncated passage time.

\begin{lemma}\label{lem_number_of_boxes}
	Let~$x \in \R^d$ with~$\|x\| \ge 1$ and~$t \in [1,\|x\|]$. Let~$\gamma$ be a path in~$\mathcal G^t$ from~$o$ to~$q^t(x)$ with~$|\gamma| \le \Kprime \|x\|$ and such that~$Y_{t,x} = \sum_{e \in \gamma} \tau^t_e$. Then, the number of boxes of~$\mathscr B$ intersected by~$\gamma$ is at most~$(3^d+1)\left(\frac{(3d + \Kprime r)\|x\|}{t} + 1\right)$.
\end{lemma}
\begin{proof}
We write~$\gamma = (\gamma_0, \ldots, \gamma_n)$, where~$n = |\gamma|$,~$\gamma_0 = o$, and~$\gamma_n = q^t(x)$. We will now define an increasing sequence~$J_0,\ldots,J_m$ of indices in the path. We first let~$J_0 := 0$. Next, assuming~$J_i$ has been defined and is smaller than~$n$, we define (with the convention that the minimum of an empty set is infinity):
\begin{align*}
	J_{i+1}:= n &\wedge \min\{j > J_i: \{\gamma_{j-1},\gamma_j\} \text{ is an extra edge}\} \\
	&\wedge \min\{j > J_i: \|\gamma_j - \gamma_{J_i}\|_\infty > t\}.
\end{align*}
We let~$m$ be the index such that~$J_m = n$ (which is the last one). For cleanliness of notation, we write
\[a_i := J_{i+1} - J_{i},\quad i \in \{0,\ldots, m-1\},\]
and
\[\mathsf y(i,k) := \gamma_{J_i + k},\quad i \in \{0,\ldots,m-1\},\; k \in \{0,\ldots,a_i\},\]
so that
\[(\mathsf y(i,0),\ldots, \mathsf y(i,a_i)) = (\gamma_{J_i},\ldots, \gamma_{J_{i+1}}).\]
By construction, for any~$i$, the set~$\{\mathsf y(i,0),\ldots, \mathsf y(i,a_i)\}$ intersects at most~$3^d + 1$ boxes of~$\mathscr B$ (that is: at most the box containing~$\mathsf y(i,0)$, the boxes that are adjacent to it in~$\ell_\infty$-norm, and the box containing~$\mathsf y(i,a_i)$). Hence,
\begin{align*}
&|\{ \text{boxes of $\mathscr B$ intersected by $\gamma$}\}|\\ &\le  \sum_{i=0}^{m-1}|\{ \text{boxes of $\mathscr B$ intersected by $\{\mathsf y(i,0),\ldots, \mathsf y(i,a_i)\}$}\}|  \le (3^d+1)m.
\end{align*}
We now want to give an upper bound for~$m$. For this, we write
	\[m = m_{\mathrm{basic}} + m_{\mathrm{extra}}+1,\]
where~$m_{\mathrm{extra}}$ is the number of~$i < m-1$ such that the last step in the sub-path~$(\mathsf y(i,0),\ldots, \mathsf y(i,a_i))$ is an extra edge, and~$m_{\mathrm{basic}}$ is the number of other ${i <m-1}$ (that is, for which the sub-path~$(\mathsf y(i,0),\ldots, \mathsf y(i,a_i))$ does not traverse any extra edge). Noting that
\[3d \K \|x\| \stackrel{\eqref{eq_bound_on_Y}}{\ge} Y_{t,x} = \sum_{e \in \gamma} \tau^t_e \ge \K t |\{e \in \gamma: e \text{ is extra}\}|,\]
we have~$m_{\mathrm{extra}} \le \frac{3d \K \|x\|}{\K t}= \frac{3d  \|x\|}{t}$. Next, if~$i$ is an index that contributes to~$m_{\mathrm{basic}}$, then
\[t \le \|\mathsf y(i,a_i) - \mathsf y(i,0)\|_\infty \le \sum_{k=0}^{a_i-1} \|\mathsf y(i,k+1) - \mathsf y(i,k)\| \le ra_i,\]
so~$a_i \ge t/r$, and then,
\[\Kprime \|x\| \ge |\gamma| = \sum_{i=0}^{m-1} a_i \ge m_{\mathrm{basic}} \cdot \frac{t}{r},\]
so~$m_{\mathrm{basic}} \le \frac{\Kprime r \|x\|}{t}$. This gives~$m \le \frac{(3d + \Kprime r)\|x\|}{t}+1$.
\end{proof}

The proposition below is based on the results established in \citet{boucheron2003}. It sets the stage for obtaining a concentration inequality for our truncated random variable $Y_{t,x}$, which will be derived later.

\begin{proposition}\label{prop_concentration}
    Let~$(S,\mathcal{S})$ be a measurable space and~$n \in \N$. Let~$\sigma$ be a probability measure on~$(S,\mathcal{S})$ and~$\p$ be the probability on~$(S^n,\mathcal{S}^n)$ given by the {$n$-fold} product measure of~$\sigma$. Set $X_1, \dots, X_n$ to be independent random elements taking values in~$S$ and let $X_i'$ be an independent copy of $X_i$. Set~$h: S^n \to\R$ to be a measurable function on $(S^n,\mathcal{S}^n,\p)$ and fix
    \[Z:=h(X_1,\dots,X_n), \ \ \text{and} \ \ Z^{(i)}:=h(X_1,\dots,X_{i-1},X_{i}',X_{i+1},\dots,X_n).\]
    Assume that there exist~$\kappa > 0$ and measurable sets~$E_1,\ldots, E_n \in \mathcal{S}^n$ such that, $\p$-a.s., for each $i \in \{1,\ldots,n\}$,
    \begin{equation} \label{eq_condition_measurable}           Z^{(i)}-Z \le \kappa \cdot \mathbbm{1}_{E_i}.
    \end{equation}
	Then, 
        \begin{equation} \label{eq:new_steele-efron-stein}
            \mathrm{Var}Z \leq \kappa^2 \sum_{i=1}^n\p(E_i).
        \end{equation}
    Furthermore, for any~$\alpha > 0$ such that $\E[e^{\alpha Z}]< +\infty$, one has
    \begin{equation} \label{eq:exp.Vminus.bound}
       \E\left[\exp\left\{\alpha (\E[Z]-Z)\right\} \right] \le \E \left[ \exp \left\{ 2 \alpha^2 \kappa^2\sum_{i=1}^n \mathbbm{1}_{E_i}\right\} \right] 
    \end{equation}
    and, if~$\alpha > 0$ is such that $\E[e^{\alpha Z}]< +\infty$, then for all $\lambda \in(0,\alpha)$,
    \begin{equation} \label{eq:exp.W.bound}
        \E\left[\exp\left\{\lambda (Z-\E[Z])\right\} \right] \le \E \left[ \exp \left\{ 2\lambda^2 \kappa^2 e^{\alpha\kappa} \sum_{i=1}^n \mathbbm{1}_{E_i}\right\} \right].
    \end{equation}
\end{proposition}
\begin{proof}
    The bounds for \eqref{eq:new_steele-efron-stein} and \eqref{eq:exp.Vminus.bound} are directly obtained by applying \eqref{eq_condition_measurable} in the Steele-Efron-Stein inequality (see \cite[p.~1585]{boucheron2003}) and in Theorem 2 of \citet{boucheron2003} with $\lambda=\alpha$ and $\theta = 1/(2\alpha)$.  Additionally, item \eqref{eq:exp.W.bound} is a specific case addressed in Lemma 3.2 of \citet{garet2010}. Below, we demonstrate how it can be attained from Theorem 2 of \cite{boucheron2003}.
    
    Let us set $\psi(s)= s(e^s-1)$ and note that $\psi(-s) \le s^2$ for $s \ge 0$. Hence, by Lemma 8 of \citet{massart2000},
    \begin{align*}
        \alpha \E[e^{\alpha Z}] -\E[e^{\alpha Z}]\log\left(\E[e^{\alpha Z}]\right) &\le  \sum_{i=1}^n\E\left[e^{\alpha Z}\psi\big(-\alpha(Z -Z^{(i)})\big)\mathbbm{1}_{Z-Z^{(i)} \ge 0}\right] \\
        &\le  \alpha^2\sum_{i=1}^n\E\left[e^{\alpha Z}\big((Z -Z^{(i)})_+\big)^2\right] \\
        &=  \alpha^2\sum_{i=1}^n\E\left[e^{\alpha Z^{(i)}} \big((Z^{(i)}-Z)_+\big)^2\right] \\
        &\le   \alpha^2\kappa^2\sum_{i=1}^n\E\left[e^{\alpha Z} e^{\alpha( Z^{(i)}-Z)} \mathbbm{1}_{E_i}\right] \\
        &\leq \E\left[ e^{\alpha Z} \left(\alpha^2 \kappa^2 e^{\alpha\kappa}\sum_{i=1}^n\mathbbm{1}_{E_i}\right)\right]
    \end{align*}
    and the remaining steps of the proof follow the same methodology as outlined in Theorem 2 of \citet{boucheron2003}.
\end{proof}

We use the inequalities obtained in the proposition above for the truncated random variable resulting in the following lemma. This lemma is the last step before proving Proposition~\ref{prop_new_concentration} (and the constant~$C_{\mathrm{dev}}$ given here is the same as the one in that proposition).

\begin{lemma}\label{lem_apply_conc}
	There exists~$C_{\mathrm{dev}} \ge 1$ such that the following holds. Let~$x \in \R^d$ with~$\|x\| \ge 1$ and~$t \in [1,\|x\|]$. Then, 
	\begin{equation*}
		\E\left[\exp \left\{ \alpha \cdot | \E[Y_{t,x}]-Y_{t,x} |\right\}\right] \le   \exp \left\{C_{\mathrm{dev}} \alpha^2 t~\|x\| \right\} \quad \text{for any }\alpha \in \left(0,\tfrac{1}{4\K t}\right).
	\end{equation*}
\end{lemma}
\begin{proof}
In order to apply Proposition~\ref{prop_concentration} to~$Y_{t,x}$, we need to define this random variable in a probability space with a suitable product measure. We achieve this by constructing~$\mathcal G$ (and consequently also~$\mathcal G^t$) in a probability space where the randomness is given in blocks, each block corresponding to a box of~$\mathscr B$. This has to be done with some care, because we need to describe how to encode the randomness corresponding to the passage times of all edges, including those that touch distinct boxes.

We take a probability space with probability measure~$\p$ where we have defined a family of independent random elements
\[
	X_i := (\mathscr V_i, (\mathscr T_i(k,\ell): k \in \mathbb N_0,\;\ell \in \mathbb N)),\quad  i \in \N,
\]
where~$\mathscr V_i$ is a Poisson point process on~$[-t/2,t/2)^d$ with intensity~$1$, and
independently for each~$k,\ell$ (and independently of~$\mathscr V_i$),~$\mathscr T_i(k,\ell)$ is a random variable with the law of the passage time~$\tau$.

Now, fix an arbitrary enumeration~$\{z_1,z_2,\ldots\}$ of~$\Z^d$. We use the above random elements to construct~$\mathcal G = (V,\mathcal{E})$ and the passage times~$T(\cdot,\cdot)$ as follows. First construct the set of vertices as
\[V= \{t z_i + u:\; i \in \N,\; u \in \mathscr{V}_i\}.\]
Next, construct the set of edges as prescribed for a random geometric graph with range parameter~$r$: an edge is included between any two vertices within Euclidean distance~$r$ of each other.

The passage times are given as follows. Let~$i,j \in \mathbb N$ with~$i \le j$; set~$k:=j-i$. In case there are edges from~$z_i+[-t/2,t/2)^d$ to~$z_j+[-t/2,t/2)^d$, let~$\mathsf e_{i,k,1},\mathsf e_{i,k,2},\ldots$ be these edges, ordered using some deterministic procedure (say, some lexicographic order). Then, set
\[
\tau_{\mathsf e_{i,k,\ell}} = \mathscr T_i(k,\ell)
\]
for each~$\ell$. Following this procedure for all pairs~$i,j$ with~$i \le j$, we obtain the passage times of all edges in the graph. This completes the construction of~$\mathcal G$, and then~$\mathcal G^t$ is obtained from it as before (by adding extra vertices and edges).

We have exhibited an infinite product space, but in fact, by the definition of~$Y_{t,x}$, there exists a deterministic constant~$ N_{t,x}$ such that we can write
\begin{equation*}
Y_{t,x} = h(X_1,\ldots, X_{N_{t,x}})
\end{equation*}
for some function~$h$.

The next step is to exhibit a constant~$\kappa$ and events~$E_i$ for~\eqref{eq_condition_measurable}. Let~$\gamma$ be a path in~$\mathcal G^t$ from~$o$ to~$q^t(x)$ such that~$|\gamma| \le \Kprime \|x\|$ and such that~$Y_{t,x} = \sum_{e \in \gamma} \tau^t_e$ (in case there are multiple paths with this property, we choose one using some arbitrary procedure). Then, let~$E_i$ be the event that~$\gamma$ intersects~$z_i + [-t/2,t/2)^d$, for~$i \in \{1,\ldots, N_{t,x}\}$. Set
\[X:=(X_1,\dots, X_{N_{t,x}}), \quad X_{(i)}:=(X_1,\dots,X_{i-1},X_i',X_{i+1},\dots, X_{N_{t,x}}),\]
and recall that $Y_{t,x}^{(i)} = h(X_{(i)})$ for any independent copy $X_i'$ of $X_i$. It is easy to check that, if~$X$ and $X_{(i)}$ are two elements of our probability space that agree in all entries except possibly the~$i$-th, then
\begin{equation} \label{eq:new_V-minus}
    Y_{t,x}^{(i)} - Y_{t,x} \le 4\K t \cdot \mathbbm{1}_{E_i}.
\end{equation}
Moreover,
	\begin{equation} \label{eq:sum.indicators_Y}
            \sum_{i=1}^{N_{t,x}} \mathbbm{1}_{E_i} \le 3^{d+1}(3d + \Kprime r)\frac{\|x\|}{t} + 3^{d+1} \le 2\cdot3^{d+1}(3d+\Kprime r)\frac{\|x\|}{t},
        \end{equation}
	where the first inequality follows from Lemma~\ref{lem_number_of_boxes}, and the second from the choice of~$t$. Now, Proposition~\ref{prop_concentration} states that, for any~$\alpha \in \left(0,~1/(4\K t)\right)$, and
		\begin{equation*}
			\E[\exp\{ \alpha \cdot |Y_{t,x} - \E[Y_{t,x}]|\}] \le   \exp \left\{32 \cdot 3^{d+1} \K^2 e~(3d + \Kprime r) \alpha^2 \cdot  t~\|x\| \right\}
	\end{equation*}
    Hence, given the conditions on $\|x\|$ and $t$, our assertion holds for $C_{\mathrm{dev}}$ sufficiently large, uniformly across all $x$ and $t$, thereby establishing the lemma.
\end{proof}

\begin{proof}[Proof of Proposition~\ref{prop_new_concentration}]
By statement of the proposition follows from combining the above lemma with Jensen's and Markov's inequality.
\end{proof}

\section{Expectation and variance bounds} \label{sec:expectation.variance}
The proof of Theorem~\ref{thm:moderate.dev.FPP} is divided into two parts. We begin by establishing the variance bound.
\begin{proof}[Proof of Theorem~\ref{thm:moderate.dev.FPP}, item \eqref{eq:VarT}]
	First, observe that, by \cref{cor_join_expectation} for all $t>0$ and sufficiently large $\|x\|$,
    \begin{eqnarray*}
        \operatorname{Var} T(x) &\leq 2 \operatorname{Var} Y_{t,x} &+ ~2\operatorname{Var} \big(Y_{t,x} - T(x)\big)\\
        &\leq 2 \operatorname{Var}Y_{t,x} &+~ 4 (\mathfrak C t)^2.
    \end{eqnarray*}
    
    By \eqref{eq:new_steele-efron-stein}, \eqref{eq:new_V-minus}, and \eqref{eq:sum.indicators_Y}, one has
    \[\operatorname{Var} Y_{t,x} \leq 16 \K^2 t^2 \cdot \E\left[\sum_{i=1}^{N_{t,x}}\mathbbm{1}_{E_i}\right] \leq \left(16 \K^2 3^{d+1}(3d+\Kprime r) \right)\cdot t~ \|x\|.\] 
    
    Let us fix $t=\mathfrak C_1 \log\big(\|x\|\big)$ for sufficiently large $\|x\|$ and it completes the proof.
\end{proof}

We now turn to the proof of the expectation bound~\eqref{eq_asymp_exp}. For this, we follow \citet{howard2001} and Yao~\cite{yao2013}. We will need the following preliminary result.
\begin{lemma} \label{lm_expectation.double}
	There exists~$\bar{C} > 0$ such that, for all sufficiently large $t>0$ and all $x \in \partial B(o, 1)$,
    \[
	    2 \E[T(t\cdot x)] \leq \E[T(2t \cdot  x)] + \bar{C} \sqrt{t} \log(t).
    \]
\end{lemma}
\begin{proof}
    Since the distribution of $T$ is rotation invariant, it suffices to verify the result for $T(t \cdot e_1)$ with $e_1 \in \partial B(o,1)$ an element of the canonical basis of $\R^d$. We take~$t \ge 0$ which will be assumed large throughout the proof.

	It is straightforward to see that, if~$t$ is large, there exist~$n_t \le t^d$ and (deterministic) points~$x_1,\ldots,x_{n_t} \in \partial B(o,t)$ such that
	\begin{equation}
		\label{eq_cover1}
		B(o,t)\backslash B(o,t-r) \subseteq \bigcup_{i=1}^{n_t} B(x_i,t^{1/4}).
	\end{equation}
	We set~$y_i:=2te_1 + x_i \in \partial B(2te_1,t)$, so that
	\begin{equation}
		\label{eq_cover2}
		B(2te_1,t)\backslash B(2te_1,t-r) \subseteq \bigcup_{i=1}^{n_t} B(y_i,t^{1/4}).
	\end{equation}
Define~$\mathsf F_t$ as the event that all of the following happen:
	\begin{itemize}
		\item[(i)] $q(o) \in B(o,t)$,
		\item[(ii)] $q(2te_1) \in B(2te_1,t)$,
		\item[(iii)] for every~$z \in \{x_1,\ldots,x_{n_t},y_1,\ldots,y_{n_t}\}$ and all~$w \in B(z,t^{1/4})$, we have~$T(z,w) \le \sqrt{t}$.
	\end{itemize}

	Now, fix a realization of the random graph and the passage times, and let~$\gamma$ be a path from~$q(o)$ to~$q(2te_1)$ that minimizes the passage time between these two vertices. On~$\mathsf F_t$, this path starts in~$B(o,t)$ and ends in~$B(2te_1,t)$. Traversing~$\gamma$ from~$q(o)$ to~$q(2te_1)$, let~$u^*$ be the first vertex in the annulus~${B(o,t)\backslash B(o,t-r)}$, and let~$v^*$ be the first vertex in the annulus~$B(2te_1,t)\backslash B(2te_1,t-r)$. Then, by~\eqref{eq_cover1} and~\eqref{eq_cover2}, there exist~$i_*$ such that~$u^* \in B(x_{i^*},\sqrt{t})$ and~$v^* \in B(y_{j^*},\sqrt{t})$. We then have
	\begin{align*}
		T(2te_1) \cdot \mathbbm{1}_{\mathsf F_t} &\ge  (T(u^*) + T(v^*,2te_1)) \cdot  \mathbbm{1}_{\mathsf F_t} \\
		&\ge (T(x_{i^*}) + T(y_{j^*},2te_1) - 2\sqrt{t}) \cdot \mathbbm{1}_{\mathsf F_t}.
	\end{align*}
	Taking the expectation and using symmetry, this gives
	\begin{equation}\label{eq_split_two_xs}
		\begin{split}
			\E[T(2te_1)] &\ge 2 \E\left[\mathbbm{1}_{\mathsf F_t} \cdot  \min_{1 \le i \le n_t} T(x_i)  \right] - 2\sqrt{t} \\
			&= 2\E\left[  \min_{1 \le i \le n_t} T(x_i)  \right] - 2 \E\left[\mathbbm{1}_{\mathsf F_t^c} \cdot  \min_{1 \le i \le n_t} T(x_i)  \right] - 2\sqrt{t}.
		\end{split}
	\end{equation}

	We will deal with the two expectations on the right-hand side above separately. Define~$\E[T(x_i)] = \E[T(te_1)] =: \nu_t$ for every~$i$. We will need the fact that there exists some~$C> 0$ such that, for~$t$ large enough,
	\begin{equation}\label{eq_fast_for_nu}
		\nu_t \le C t.
	\end{equation}
	This can be easily obtained from~\cref{prop:Hn.growth}.

	Let us give a lower bound for the first expectation on the right-hand side of~\eqref{eq_split_two_xs}. Using Theorem~\ref{thm_new_moderate_deviations}, we can choose~$C > 0$ such that, for~$t$ large enough, we have~$\p(T(te_1) \le \nu_t - C\sqrt{t}\log(t)) < 1/t^{d+1}$. We then bound
	\begin{align*}
		\E\left[  \min_{1 \le i \le n_t} T(x_i)  \right] &\ge (\nu_t - C\sqrt{t}\log(t))\cdot  \p\left(\min_{1\le i \le n_t}  T(x_i) > \nu_t - C \sqrt{t} \log(t)\right)\\
		&\ge (\nu_t - C\sqrt{t}\log(t))\cdot (1-n_t\cdot \p(T(te_1) \le \nu_t - C\sqrt{t}\log(t)))\\
		&\ge (\nu_t - C\sqrt{t}\log(t))\cdot (1-t^d/t^{d+1}).
	\end{align*}
	Using~\eqref{eq_fast_for_nu}, we see that the right-hand side above is larger than~$\nu_t - C'\sqrt{t}\log(t)$ for some constant~$C' > 0$ and~$t$ large enough.

	To give an upper bound for the second expectation on the right-hand side of~\eqref{eq_split_two_xs}, we first bound
	\[\E\left[\mathbbm{1}_{\mathsf F_t^c} \cdot  \min_{1 \le i \le n_t} T(x_i)  \right] \le \p(\mathsf F_t^c)^{1/2}\cdot  \E[T(te_1)^2]^{1/2}\]
	using Cauchy-Schwarz. Next, by~\cref{prop:Hn.growth,lm:T.bds.sup}, we have
	\begin{align*}
		\p(\mathsf F_t^c) &\le 2 \p(\mathcal H \cap B(o,t) = \varnothing) + n_t\cdot \p\left(\max_{w \in B(o,t^{1/4})} T(o,w) > \sqrt{t}\right)\\
		&\le 2Ce^{-ct} + t^d \cdot C e^{-ct^{1/4}}
	\end{align*}
    for some~$C,c> 0$. We also bound
	\begin{align*}
		\E[T(te_1)^2] = \mathrm{Var}(T(te_1)) + \E[T(te_1)]^2 \le Ct^2
	\end{align*}
	for some~$C > 0$, by~\eqref{eq:VarT} and~\eqref{eq_fast_for_nu}. 
	
	Putting things together, we have proved that the right-hand side of~\eqref{eq_split_two_xs} is larger than
	\begin{align*}
		&2(\nu_t - C'\sqrt{t} \log(t)) - 2\cdot (2Ce^{-ct} + t^d \cdot C e^{-ct^{1/4}})^{1/2} \cdot (Ct^2)^{1/2} -2\sqrt{t} \\[.2cm]
		&\ge 2\nu_t - C'' \sqrt{t}\log(t)
	\end{align*}
	for some~$C'' > 0$ and~$t$ large enough.
\end{proof}

We can now conclude the proof of Theorem~\ref{thm:moderate.dev.FPP} by applying the lemmas above.

\begin{proof}[Proof of Theorem~\ref{thm:moderate.dev.FPP}, item \eqref{eq_asymp_exp}]
	Recall the isotropic properties and subadditivity of $T(x)$. By Kingman's subadditive ergodic theorem
    \[
        \lim_{n \uparrow +\infty}\frac{T(nx)}{n} = \lim_{n \uparrow +\infty}\frac{\E\big[T(nx)\big]}{n} = \inf_{n \in \N}\frac{\E\big[T(nx)\big]}{n} = \phi(x)
    \]
	where $\phi(x)= \|x\|/\upvarphi$ with $\upvarphi\in(0,+\infty)$ (see \cite[Eq.~(3.6)]{coletti2023} for details). %Note that this already gives the first inequality in~\eqref{eq_asymp_exp}.

	Since $\E\big[T(x)\big] = \E\big[T(te_1)\big]$ for $t=\|x\|$, it suffices to prove the result for $\E\big[T(te_1)\big]$ with large $t>0$. Let $\bar{C}>0$ be the constant of Lemma~\ref{lm_expectation.double}. One can obtain the result by applying Lemma 4.2 of \citet{howard2001}. To be self-contained and more precise, let us define
	\[f(t):=\E\big[ T(te_1) \big]\quad  \text{and} \quad \Tilde{f}(t):= f(t)-10\bar{C}\sqrt{t}\log(t), \quad t > 0.\]
	We have
	\[\Tilde{f}(2t) = f(2t) - 10 \bar{C} \sqrt{2t} \log(2t) \ge 2f(t) - \bar{C}\sqrt{t}\log(t) - 10\bar{C}\sqrt{2t}\log(2t),\]
	where the inequality holds for~$t$ large enough, by~Lemma~\ref{lm_expectation.double}. When~$t$ is large we have~$10\sqrt{2t}\log(2t) \le 19\sqrt{t}\log(t)$ (since~$10\sqrt{2} < 19$ and~$\log(2t)/\log(t) \to 1$ as~$t \uparrow +\infty)$, so we obtain
	\[\Tilde{f}(2t) \ge 2 f(t) -  \bar{C}\sqrt{t}\log(t) - 19\bar{C}\sqrt{t}\log(t) = 2f(t) - 20\bar{C}\sqrt{t}\log(t) = 2\Tilde{f}(t).\]
	Iterating~$\Tilde{f}(2t) \ge 2 \Tilde f(t)$, we get
    \[
        \frac{\Tilde{f}(2^n t)}{2^n t} \geq \frac{\Tilde{f}(t)}{t} \quad \text{for all large } t \text{ and all }n \in \N.
    \]
    Taking a limit of the left-hand side as~$n \uparrow +\infty$ we obtain that, for all large~$t$,
	\[ \frac{1}{\upvarphi} \ge \frac{\Tilde{f}(t)}{t} = \frac{\E\big[T(t\cdot e_1)\big] - 10 \bar{C} \sqrt{t}\log(t)}{t},\]
    which yields the desired conclusion.
\end{proof}

\section{Quantitative shape theorem}\label{sec:quantitative.shape}

Building upon the results established in Theorems \ref{thm_new_moderate_deviations} and \ref{thm:moderate.dev.FPP}, this section focuses on the speed of convergence of the asymptotic shape presented in Theorem \ref{thm:speed.FPP}.

We will need the following auxiliary result.

\begin{lemma}\label{lem_new_c_star}
    Under the assumptions of Theorem~\ref{thm:speed.FPP}, there exists~$C^* > 0$ such that almost surely, for every~$x \in \mathbb R^d$ with~$\|x\|$ large enough, we have
    \begin{equation}\label{eq_choice_of_R}
    \left| T(x) - \frac{\|x\|}{\upvarphi}\right| \le C^* \sqrt{\|x\|}\log\big(\|x\|\big).
    \end{equation}
\end{lemma}
\begin{proof}
    By putting together Theorem~\ref{thm_new_moderate_deviations} and~\eqref{eq_asymp_exp}, we can choose~$C^* > 0$ such that, for all~$x \in \mathbb R^d$ with~$\|x\|$ large enough,
\begin{equation} \label{eq_before_fubini}
	\mathbb P\left(\left|T(x)-\frac{\|x\|}{\upvarphi}\right| > \frac{C^*}{2} \sqrt{\|x\|} \log(\|x\|)\right) < \|x\|^{-5d}.
\end{equation}

Define
\[Z(x):= \mathds{1}\left\{ \left| T(x) - \frac{\|x\|}{\upvarphi}\right| > C^* \sqrt{\|x\|} \log(\|x\|) \right\}.\]
Our goal is to prove that
\begin{equation}\label{eq_goal_Z}
	\mathbb P \left(\limsup_{\|x\| \to \infty} Z(x) = 1 \right) = 0.
\end{equation}

Let us start by defining the auxiliary random variables
\[\hat Z(x):= \|x\|^{3d} \cdot \mathds{1}\left\{ \left| T(x) - \frac{\|x\|}{\upvarphi}\right| > \frac{C^*}{2} \sqrt{\|x\|} \log(\|x\|) \right\}.\]
By Fubini's theorem and~\eqref{eq_before_fubini}, we have
\begin{equation*}
	\mathbb E\left[\int_{\mathbb R^d} \hat Z(x)\;\mathrm{d}x \right] = \int_{\mathbb R^d }\|x\|^{3d} \cdot \mathbb P\left(\left| T(x) - \frac{\|x\|}{\upvarphi}\right| > \frac{C^*}{2} \sqrt{\|x\|} \log(\|x\|)\right)\;\mathrm{d}x < \infty,
\end{equation*}
hence
\begin{equation}\label{eq_hat_Z_finite}
	\int_{\mathbb R^d} \hat Z(x)\; \mathrm{d}x < \infty \text{ almost surely.}
\end{equation}

Using Proposition~\ref{prop:holes.H}, it is easy to show that there exists an almost surely finite random variable~$R_1$ such that
\begin{equation*}
	\|q(x)-x\| \le \log^2(\|x\|)\quad \text{for every } x \in B(o,R_1)^c.
\end{equation*}
Next, using for instance Mecke's formula, we can show that there exists an almost surely finite random variable~$R_2$ such that
\begin{equation*}
	B(x,\|x\|^{-2}) \cap V = \{x\} \quad \text{for every } x \in V \cap B(o,R_2)^c.
\end{equation*}
That is, every~$x$ in the point process with norm above~$R_2$ is at distance more than~$\|x\|^{-2}$ from any other point of the point process.

Fix a realization of the random graph and the passage times for which~$R:=\max(R_1,R_2)< \infty$. We will show that, for~$x \in \mathbb R^d$,
\begin{equation}\label{eq_want_hat_Z}
	\text{if }\|x\| \text{ is large enough and } Z(x)=1, \text{ then }\int_{B(x,2\log^2(\|x\|))} \hat{Z}(y)\;\mathrm{d}y \ge 1.
\end{equation}
This will imply that
\begin{equation*}
	\text{if } \limsup_{\|x\|\to \infty} Z(x) = 1, \text{ then } \int_{\mathbb R^d} \hat{Z}(y)\;\mathrm{d}y = \infty.
\end{equation*}
By putting together this with~\eqref{eq_hat_Z_finite}, we obtain~\eqref{eq_goal_Z}.

We now turn to the proof of~\eqref{eq_want_hat_Z}. Let~$x \in \mathbb R^d$ be a point with~$Z(x)=1$, so that
\begin{equation}\label{eq_what_for_Tx}
	\left| T(x) - \frac{\|x\|}{\upvarphi}\right| > C^* \sqrt{\|x\|}\log(\|x\|).
\end{equation}
We will repeatedly assume that $\|x\|$ is large enough. To start, assume that ${\|x\| \ge R_1}$, so that we have
\begin{equation}\label{eq_r1_again}\|q(x)-x\|< \log^2(\|x\|),\end{equation}
	by the definition of~$R_1$. We then increase~$\|x\|$ further so that ${\|x\|-\log^2(\|x\|)> R_2}$; then, using~\eqref{eq_r1_again} and the definition of~$R_2$, we obtain
	\begin{equation*}
V \cap B(q(x),\|q(x)\|^{-2}) = \{q(x)\}.
	\end{equation*}
This implies that
\begin{equation}\label{eq_implies_that}
	\text{for all }y \in B(q(x),\tfrac12\|q(x)\|^{-2}), \text{ we have } q(y) = q(x),\text{ so } T(y) = T(x).
\end{equation}
Using~\eqref{eq_r1_again}, we increase~$\|x\|$ further to guarantee that
\begin{equation}\label{eq_implies_also}
B(q(x),\tfrac14\|x\|^{-2}) \subseteq B(q(x),\tfrac12\|q(x)\|^{-2}) \subseteq B(x,2\log^2(\|x\|)).
\end{equation}
Finally, when~$\|x\|$ is large, for every $y$ with~$\|y-x\| \le 2\log^2(\|x\|))$ we have
\begin{equation*}
	\left| \frac{\|x\|}{\upvarphi} - \frac{\|y\|}{\upvarphi}\right| \le \frac{C^*}{4}\sqrt{\|x\|}\log(\|x\|)
\end{equation*}
and
\begin{equation*}
	\left|  \sqrt{\|x\|} \log(\|x\|) -  \sqrt{\|y\|} \log(\|y\|) \right| \le \frac{1}{4}\sqrt{\|x\|}\log(\|x\|).
\end{equation*}
Consequently, for every~$y$ with~$\|y-x\|\le 2\log^2(\|x\|)$, we have
\begin{equation}\label{eq_c_star_inclusion}\begin{split}
	&\left[ \frac{\|y\|}{\upvarphi} - \frac{C^*}{2} \sqrt{\|y\|}\log(\|y\|), \frac{\|y\|}{\upvarphi} + \frac{C^*}{2} \sqrt{\|y\|}\log(\|y\|) \right]  \\
&\subseteq\left[ \frac{\|x\|}{\upvarphi} - {C^*} \sqrt{\|x\|}\log(\|x\|), \frac{\|x\|}{\upvarphi} + {C^*} \sqrt{\|x\|}\log(\|x\|) \right].
\end{split}
\end{equation}

Now, fix~$y \in B(q(x),\tfrac14 \|x\|^{-2})$. By~\eqref{eq_implies_that} and~\eqref{eq_implies_also}, we have~$T(y)=T(x)$. Further, comparing~\eqref{eq_what_for_Tx} and~\eqref{eq_c_star_inclusion} shows that
\[
	\left| T(y)-\frac{\|y\|}{\upvarphi} \right| > \frac{C^*}{2} \sqrt{\|y\|} \log(\|y\|).
\]
Using~$\|y\| \ge \|q(x)\|-\|q(x)\|^{-2} \ge \|x\|/2$, we then obtain~$\hat Z(y) \ge 2^{-3d} \|x\|^{3d}$. Then,
\[
	\int_{B(x,2\log^2(\|x\|))}\hat Z(y)\;\mathrm{d}y \ge 2^{-3d} \int_{B\left(x,\tfrac14 \|x\|^{-2}\right)} \|x\|^{3d} \;\mathrm{d}x > 1
\]
when~$\|x\|$ is large, completing the proof of~\eqref{eq_want_hat_Z}.

\end{proof}

\begin{proof}[Proof of Theorem~\ref{thm:speed.FPP}]

Let~$C^*$ be the constant of Lemma~\ref{lem_new_c_star}.
	Writing
	\[g_+(a):= \frac{a}{\upvarphi} + C^* \sqrt{a} \log(a),\quad g_-(a):= \frac{a}{\upvarphi} - C^* \sqrt{a} \log(a),\quad a > 0,\]
	the inequality~\eqref{eq_choice_of_R} can be written as
	\begin{equation}
		\label{eq_becomes}
	g_-\big(\|x\|\big) \le T(x) \le g_+\big(\|x\|\big).
	\end{equation}

	Fix a realization of the random graph and passage times, and let~$R$ be large enough that~\eqref{eq_becomes} holds for all~$x$ with~$\|x\| \ge R$. We can take~$R' \ge R$ such that~$g_-$ is increasing on~$[R',\infty)$.  Now fix~$t$ large enough that
	\begin{equation*}
		t \ge \max\left\{\max_{x \in B(o,R')} T(x),\; \frac{R'}{\upvarphi}\right\}.
	\end{equation*}
We will now prove that, letting~$b := 4\upvarphi \sqrt{\upvarphi} C^*$, we have, $\p$-almost surely,
	\begin{equation}\label{eq_two_inclusions} 
		B(o, \upvarphi t - b \sqrt{t}\log(t)) \subseteq H_t \subseteq B(o, \upvarphi t + b \sqrt{t}\log(t)).
	\end{equation}

	To prove the first inclusion, we fix
	\begin{equation}
		\label{eq_my_choice_x}
	x \in B(o,\upvarphi t - b \sqrt{t}\log(t))
	\end{equation}
	and show that~$T(x) \le t$. This inequality is automatic from the choice of~$t$ if~$x \in B(o,R')$, so assume that~$\|x\| > R'$. Then, since~$g_+$ is increasing,
	\begin{align*}
		T(x) &\stackrel{\eqref{eq_becomes}}{\le} g_+\big(\|x\|\big)  \stackrel{\eqref{eq_my_choice_x}}{\le} g_+(\upvarphi t - b\sqrt{t}\log(t)) \\[.2cm]
		&= t - \frac{b}{\upvarphi} \sqrt{t}\log(t)+ C^* \sqrt{\upvarphi t - b\sqrt{t} \log(t)} \cdot \log(\upvarphi t - b \sqrt{t}\log(t)).
		 	\end{align*}
	Increasing~$t$ if necessary, we have
	\[\sqrt{\upvarphi t - b\sqrt{t} \log(t)} \cdot \log(\upvarphi t - b \sqrt{t}\log(t)) \le 2\sqrt{\upvarphi} \cdot \sqrt{t}\log(t),\]
so we get
	\[T(x)\le t - \frac{b}{\upvarphi}\sqrt{t} \log(t) + 2C^* \sqrt{\upvarphi} \cdot \sqrt{t} \log(t) \le t\]
	by the choice of~$b$.

	To prove the second inclusion in~\eqref{eq_two_inclusions}, fix~$x \notin B(o,\upvarphi t + b \sqrt{t}\log(t))$, and let us show that~$T(x) > t$. We start with
	\[T(x) \ge g_-\big(\|x\|\big).\]
	Since~$\|x\| \ge \upvarphi t + b\sqrt{t}\log(t) \ge R'$, and~$g_-$ is non-increasing in~$[R',\infty)$, we have
	\begin{align*}
		g_-\big(\|x\|\big) &\ge g_-(\upvarphi t + b\sqrt{t}\log(t))\\[.2cm]
		&= t + \frac{b}{\upvarphi} \sqrt{t}\log(t)- C^* \sqrt{\upvarphi t + b\sqrt{t} \log(t)} \cdot \log(\upvarphi t + b \sqrt{t}\log(t)).
	\end{align*}
Increasing~$t$ if necessary, we have
	\[\sqrt{\upvarphi t + b\sqrt{t} \log(t)} \cdot \log(\upvarphi t + b \sqrt{t}\log(t)) \le 2\sqrt{\upvarphi} \cdot \sqrt{t}\log(t),\]
	so
	\[T(x) \ge t + \frac{b}{\upvarphi} \sqrt{t}\log(t) - 2C^*\sqrt{\upvarphi} \sqrt{t}\log(t) \ge t\]
	again by the choice of~$b$, completing the proof.
\end{proof}

\section{Fluctuations of the geodesics and spanning trees} \label{sec:fluctuations.spanning.tree}

Having established the framework of moderate deviations, we now investigate the probabilistic behaviour of geodesic paths. Specifically, we leverage the derived properties to characterize the probabilities of geodesics residing in specific regions of  space. Additionally, we gain insights regarding the spanning trees generated by the passage times. Throughout this section, for~$x \in \R^d$ and~$A \subseteq \R^d$, we write
\[
\mathrm{dist}(x,A):=\inf\{\|x-y\|:y \in A\}.
\]

\begin{proof}[Proof of \cref{thm:fluctuations.of.geodesics}]
Fix~$\epsilon > 0$ and~$x,y \in \mathbb R^d$; throughout the proof, we will assume that~$\|x-y\|$ is large enough; by adjusting constants at the end, we can take care of the case where~$\|x-y\|$ is small. We also fix~$R > 0$  to be chosen later, with~$1 \ll R \ll \|x-y\|$.

Define the sets
\begin{align*}
&\mathsf A := \{v \in \mathbb R^d:\mathrm{dist}(v,[\overline{xy}]) \le R\},\\
&\mathsf A' := \{v \in \mathbb R^d:R \le \mathrm{dist}(v,[\overline{xy}]) \le R+r\}.
\end{align*}
It will be useful to observe that
\begin{equation}\label{eq_plane_geo}
\text{if } v \in \mathsf A', \text{ then } \|x-v\| + \|y-v\| \ge 2 \left( \frac{\|x-y\|^2}{4} + R^2\right)^{1/2}.
\end{equation}
This follows from elementary observations in plane geometry (we can focus on the plane containing~$x,y,v$), and we sketch the proof in Figure~\ref{fig:cylindrical_region}.

\begin{figure}[htb!]
    \centering
    \includegraphics[width=0.7\linewidth]{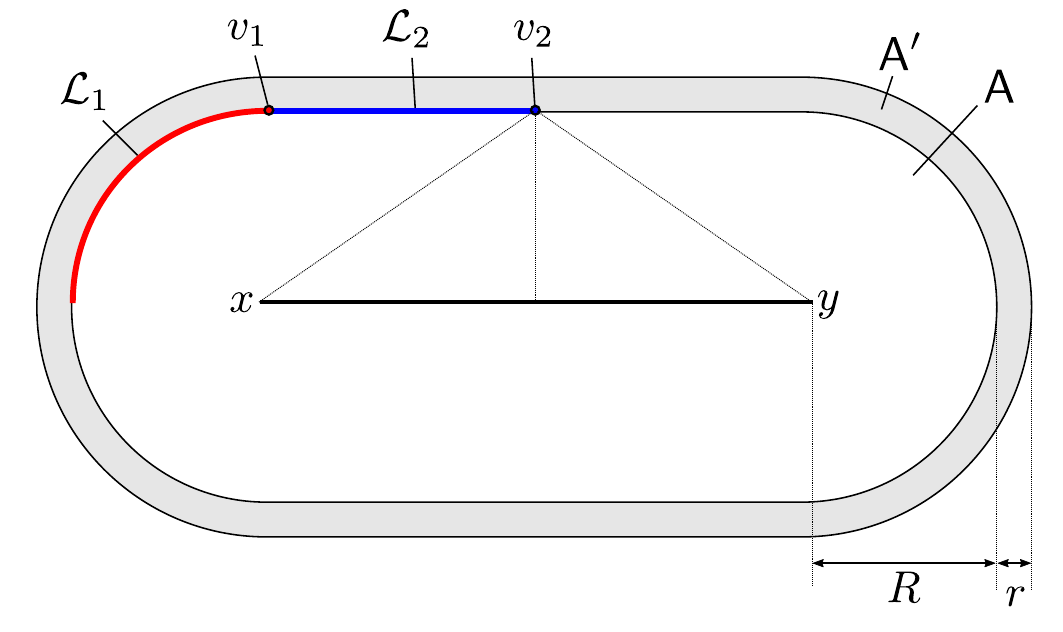}
    \caption{For~$v$ ranging over the gray region~$\mathsf A'$, the value of~$f(v):=\|x-v\|+\|y-v\|$  is minimized by~$v=v_2$. To see this, first note that, along the red curve~$\mathcal L_1$, $f(v)$ is minimized by~$v_1$ (since~$\|x-v\|=R$ along this curve, and~$\|y-v\|$ clearly increases as we move~$v$ towards the left). Next, the problem of minimizing~$f(v)$ for~$v$ along the blue line~$\mathcal L_2$ is the well-known Calculus exercise where we seek the triangle of smallest perimeter, when the base and the height are kept fixed; the minimizer is the isosceles triangle.
    }
    \label{fig:cylindrical_region}
\end{figure}

Let~$\upgamma$ be a geodesic  from~$x$ to~$y$. Using the definition of the Hausdorff distance, it is easy to check that~$d_H(\overline{\upgamma},\overline{xy}) \le R$ if and only if~$\upgamma$  is contained in $\mathsf A$. When~$\upgamma$ is \emph{not} contained in~$\mathsf A$, we consider two possibilities:
\begin{itemize}
\item $\upgamma$ is entirely contained in~$(\mathsf A \cup \mathsf A')^c$. This implies that~$\|q(x)-x\| > R+r$ and~$\|q(y)-y\| > R+r$;
\item $\upgamma$ intersects both~$\mathsf A \cup \mathsf A'$ and the complement of this set. In this case, using the fact that each jump of~$\upgamma$ has length at most~$r$ in Euclidean distance, we see that~$\upgamma$ necessarily visits a point~$v \in \mathsf A'$, and we have~$T(x,y)=T(x,v)+T(y,v)$.
\end{itemize}
Based on the above considerations, we arrive at the inclusion:
\begin{equation} \label{eq_start_plane}\begin{split}
&\left\{\exists \text{ geodesic $\upgamma$ from $x$ to $y$ with $d_H(\overline{\upgamma},\overline{xy}) > R$}  \right\} \\
& \subseteq \{\|q(x)-x\|>R\} \cup \{|q(y)-y\|>R\} \\
&\hspace{2cm}\cup\left\{\exists v \in \mathsf A': T(x,v) + T(y,v) = T(x,y) \right\}.
\end{split}
\end{equation}

We already know that~$\{\|q(x)-x\|>R\}$ and~$\{\|q(y)-y\|>R\}$ are unlikely (since~$R$ is large).
We now want to argue that the existence of~$v \in \mathsf A'$ with~$T(x,y)=T(x,v)+T(y,v)$ is also unlikely.  With this in mind, we fix~$v \in \mathsf A'$, and prove now that (if~$R$ is appropriately chosen) we have
\begin{equation}
\label{eq_key_inclusion}
\left\{\begin{array}{l} \upvarphi \cdot T(x,v) \ge \|x-v\| - \|x-v\|^{\frac12 + \epsilon},\\  \upvarphi \cdot T(y,v) \ge \|y-v\| - \|y-v\|^{\frac12 + \epsilon}, \\ \upvarphi \cdot T(x,y) \le \|x-y\| + \|x+y\|^{\frac12 + \epsilon} \end{array} \right\} \subseteq \left\{T(x,v)+T(y,v) > T(x,y) \right\}.
\end{equation}
Assume the event on the left-hand side occurs. Using~~\eqref{eq_plane_geo}, we have
\begin{equation}\label{eq_plane_inter}\begin{split}
&\upvarphi \cdot (T(x,v)+T(y,v)) \\
&\ge 2 \left(\frac{\|x-y\|^2}{4} + R^2\right)^{1/2} - \|x-v\|^{\frac12 + \epsilon} - \|y-v\|^{\frac12 + \epsilon}.
\end{split}
\end{equation}
Next, using the bound~$\sqrt{a+b} \ge \sqrt{a} + \frac{b}{2\sqrt{a+b}}$, which follows from the Mean Value Theorem, as well as~$R \ll \|x-y\|$, we have
\[
\left(\frac{\|x-y\|^2}{4} + R^2\right)^{1/2} \ge \frac{\|x-y\|}{2} + \frac{R^2}{( \|x-y\|^2/4 + R^2)^{1/2}} \ge \frac{\|x-y\|}{2} +  \frac{R^2}{\|x-y\|}.
\]
Again using~$R \ll \|x-y\|$, we also have
\[ \|x-v\| \vee \|y-v\| \le 2\|x-y\|. \]
Plugging these bounds back in~\eqref{eq_plane_inter}, we obtain
\begin{equation} \label{eq_to_compare}
\upvarphi \cdot (T(x,v)+T(y,v)) \ge \|x-y\| + \frac{2R^2}{\|x-y\|} - 2(2\|x-y\|)^{\frac12 + \epsilon}.
\end{equation}
We are also assuming that
\begin{equation}\label{eq_recap_T}
\upvarphi \cdot T(x,y) \le \|x-y\| + \|x-y\|^{\frac12 + \epsilon},
\end{equation}
so, comparing the right-hand sides of~\eqref{eq_to_compare} and~\eqref{eq_recap_T}, we see that the desired inequality~$T(x,v)+T(y,v) > T(x,y)$ is guaranteed as soon as
\[
\frac{2R^2}{\|x-y\|} - 2(2\|x-y\|)^{\frac12 + \epsilon} > \|x-y\|^{\frac12 + \epsilon}.
\]
We now see that taking
\[R:= \|x-y\|^{\frac34+\epsilon},\]
with~$\|x-y\|$ large enough, the above is satisfied (and this is essentially the smallest possible choice for~$R$, in order of magnitude).  We have now finished the proof of~\eqref{eq_key_inclusion}.

Now, using~\eqref{eq_start_plane} and~\eqref{eq_key_inclusion}, we see that
\begin{align*}
&\mathbb P(\exists \text{ geodesic $\upgamma$ from $x$ to $y$ with $d_H(\overline{\upgamma},[\overline{xy}]) > \|x-y\|^{\frac34+\epsilon}$}) \\
&\le \mathbb P(\|q(x)-x\|>R) + \mathbb P(\|q(y)-y\|>R)\\
 &\quad +\mathbb P(\upvarphi \cdot T(x,y) > \|x-y\| + \|x-y\|^{\frac12 + \epsilon})\\
&\quad +  \mathbb P( \exists v \in \mathsf A':\upvarphi \cdot T(x,v) < \|x-v\| - \|x-v\|^{\frac12 + \epsilon})\\
&\quad + \mathbb P( \exists v \in \mathsf A':\upvarphi \cdot T(y,v) < \|y-v\| - \|y-v\|^{\frac12 + \epsilon}).
\end{align*}
By Proposition~\ref{prop:Hn.growth}, the two first probabilities on the right-hand side are smaller than~$\exp\{-cR^{d-1}\}$ ($c$ is a constant that may change from line to line in what follows). By Theorems~\ref{thm_new_moderate_deviations} and~\ref{thm:moderate.dev.FPP}, the third probability is smaller than~$\exp\{-c\|x-y\|^\epsilon\}$. We now turn to the fourth and fifth probabilities.
We note that the Lebesgue measure of~$\mathsf A'$ is polynomial in~$\|x-y\|+R$, and that~$\|x-v\| \wedge \|y-v\| \ge R$ for all~$v \in \mathsf A'$. Using these considerations, Theorems~\ref{thm_new_moderate_deviations} and~\ref{thm:moderate.dev.FPP}, and Mecke's formula, it can be shown that both the fourth and the fifth probabilities are bounded from above by~$\exp\{-c\|x-y\|^{\varepsilon} \}$.
\end{proof}

Before proving \cref{thm:f.straight.spanning.trees}, we state the following lemma, which is essentially Lemma 2.7 of \cite{howard2001} (by this we mean that, although the statement of Lemma 2.7 of~\cite{howard2001} is weaker, the proof there actually gives what we state below).

\begin{lemma} \label{lm:3.7.howard.newman}
    For any~$d \ge 2$ and~$\uppsi \in (0,1/4)$, there exists~$c_\uppsi > 0$ such that the following holds. Let~$\{q_1,q_2,\ldots\}$ be a finite or infinite sequence of elements of $\mathbb R^d$ so that there exists~$k^* \in \mathbb N$ such that the following three properties hold: first,
    \[
    \|q_{i+1}-q_i\|\le \|q_i\|^{3/4} \text{ for all }i \ge k^*;
    \]
    second, 
    \[
    \|q_i\| \ge 3^{1/\uppsi} \text{ for all }i \ge k^*;
    \]
    and third,
    \[\mathrm{dist}(q_i,\overline{o q_j})\leq \|q_j\|^{1-\uppsi} \ \ \text{ for all $i,j$ with $j \ge k^*$ and ~$i \le j$}.\]
    Then, for all~$i,j$ with~$j > i \ge k^*$, we have
    \begin{equation*} \label{eq:angle.theta.ineq}
        \uptheta(q_i,q_j) \le c_\uppsi~\|q_i\|^{-\uppsi}.
    \end{equation*}
\end{lemma}

\begin{proof}[Proof of \cref{thm:f.straight.spanning.trees}]
    Fix~$\epsilon \in (0,1/4)$. Define
    \[\uppsi:=\frac14 - \frac{\epsilon}{2}, \qquad L:=\max\big\{3^{1/\uppsi}, ~r^{4/3}, ~c_\uppsi^{2/\epsilon}\big\},\] where~$c_\uppsi$ is the constant given in the above lemma.
    Define the random set
    \begin{align*}
    \mathsf Y_x:= &(\mathcal H \cap B(x, L)) \cup \left\{ \begin{array}{l} v \in \mathcal H: \text{ there is a geodesic~$\upgamma$ from~$x$ to~$v$}\\ \text{such that $d_H(\upgamma,\overline{xv}) \ge \|v-x\|^{-\frac14+\epsilon}$} \end{array}\right\}.
    \end{align*}
Then, let
\[
\mathlcal F_x := \{u \in \mathcal H: \; \exists v \in  \mathsf Y_x: \; u \text{ belongs to a geodesic from $x$ to $v$}\}.
\]
It is a straightforward consequence of Lemma~\ref{lm:3.7.howard.newman} that, if~$u \in  \mathcal H \backslash \mathlcal F_x$, and~$\upgamma$ is a geodesic that starts at~$q(x)$, first visits~$u$ and subsequently visits another vertex~$v \in \mathcal H$, then~$\uptheta(u-x,v-x) \le \|u-x\|^{-\frac14 + \epsilon}$. Hence, it suffices to prove that~$\mathlcal F_x$ is almost surely finite.

We first prove that~$\mathsf Y_x$ is almost surely finite. The set~$\mathcal H \cap B(x,L)$ is almost surely finite, since it contains at most all points of the Poisson point process inside~$B(x,L)$.  That the second set in the union that defines~$\mathsf Y_x$ is almost surely finite follows from Theorem~\ref{thm:fluctuations.of.geodesics}, Mecke's formula, and integration over~$\mathbb R^d$.

It is a straightforward consequence of the shape theorem that almost surely, for every~$v \in \mathcal H$, there are at most finitely many geodesic paths from~$x$ to~$v$. Hence, the finiteness of~$\mathlcal F_x$ follows from the finiteness of~$\mathsf Y_x$.
\end{proof}

The corollary naturally arises as a straightforward consequence of the results above, with the application of a proposition found in Howard and Newman \cite{howard2001}.

\begin{proof}[Proof of \cref{cor:asymp.dir}]
    It follows from \cref{prop:Hn.growth} that any spanning tree $\mathlcal{T}_x$ of $\Hh$ is locally finite. By \cref{prop:holes.H}, the spherical holes $\mathcal{D}(s) \in \mathcal{O}\big(\log(s)\big)$ $\p$-a.s.\ as $s \uparrow +\infty$. Thus $\Hh$ is asymptotic omnidirectional, \textit{i.e.}, for all $n \in \N$, 
    \[\left\{\frac{v}{\|v\|} \colon v \in \Hh \setminus B(o,n) \right\} \text{ is dense in } \partial B(o,1).\]
    
    Since \cref{thm:f.straight.spanning.trees} implies that $\mathlcal{T}_x$ is  $\p$-a.s. $f_\epsilon$-straight for all $\epsilon \in (0, 1/4)$ with $f_\epsilon(s)= s^{\epsilon-\frac{1}{4}}$, \cref{cor:asymp.dir} is a straightforward application of Proposition 2.8 of \citet{howard2001}.
\end{proof}

\section*{Acknowledgements}
This research was supported by grants \#2019/19056-2, \#2020/12868-9, and \#2024/06021-4, S\~ao Paulo Research Foundation (FAPESP).  Additionally, L.R. de Lima acknowledges the Department of Mathematics at the Bernoulli Institute, University of Groningen, and the University of Buenos Aires, as well as the Department of Statistics at the University of Warwick, where the foundations of \cite{delima2024} and, consequently, this article were developed during his visit. Their hospitality is greatly appreciated. Special thanks are extended to Pablo Groisman and Cristian F. Coletti for inspiring discussions on the subject matter. 

We also thank the anonymous referees for their valuable feedback and suggestions, which significantly improved the quality of this work.

\appendix
\section{Proof of approximation bounds}\label{s_appendix_proofs}

This section is dedicated to proving \cref{lem_T_t_and_T}. We assume the conditions specified in the lemma. 
We first state two auxiliary lemmas, then we give the proof of \cref{lem_T_t_and_T}, and finally we prove the two auxiliary lemmas.

\begin{lemma}\label{eq_new_hopcount}
	There exists~$c > 0$ such that, for any~$x \in \R^d$ with~$\|x\|$ large enough and~$t$ large enough (not depending on~$x$) with~$t \le \|x\|$, we have
	\begin{align*}
		\p \left(\begin{array}{l} \text{any path in $\mathcal G^t$ from $q(o)$ to $q(x)$ that minimizes}\\ \text{the $T^t$-passage time is contained in $B(o,\|x\|^3)$} \end{array} \right) > 1- e^{-ct}.
	\end{align*}
\end{lemma}

\begin{lemma}
	\label{lem_hops}
	Let~$\gamma = (x_0,\ldots, x_n)$ be a path in~$\mathcal G^t$. Assume that~$x_0$ and~$x_n$ belong to~$\mathcal H$ (the infinite cluster of~$\mathcal G$), and~$x_1,\ldots,x_{n-1}$ do not belong to~$\mathcal H$ (so that each of them is either in~$t\Z^d$ or in some finite cluster of~$\mathcal G$). Let
	\[b:= \max\left\{\|y-x\|_\infty:\; \begin{array}{l} x, y \in V,\; x \text{ and } y \text{ are in the same finite} \\ \text{cluster of }\mathcal G, \text{ and this cluster is visited by } \gamma \end{array}\right\},\]
		that is,~$b$ is the maximum~$\ell_\infty$-diameter of all finite clusters of~$\mathcal G$ intersected by~$\gamma$. Then,
	\begin{equation*}
		\sum_{e \in \gamma} \tau^t_e \ge \frac{\K t}{2t+b}\cdot \|x_n -x_0\|_\infty.
	\end{equation*}
\end{lemma}

\begin{proof}[Proof of Lemma~\ref{lem_T_t_and_T}]
	Fix~$x$ and~$t$ large enough, as required by Lemma~\ref{eq_new_hopcount}. Define the events
		\begin{align*}
			&E_1 := \left\{ \begin{array}{l} \text{any path in $\mathcal G^t$ from $q(o)$ to $q(x)$ that minimizes}\\ \text{the $T^t$-passage time is contained in $B(o,\|x\|^3)$} \end{array} \right\}, \\[.2cm]
				&E_2 := \left\{ \begin{array}{l} \text{for any } z \in \mathcal H \cap  B(o,\|x\|^3) \text{ and } y \in \mathcal{H}\cap B(z,t),\\ \text{ we have } \quad T(z,y) \le \thickbar{\beta}\cdot t \end{array} \right\},\\[.2cm]
					&E_3 := \left\{\begin{array}{l}\text{any finite cluster of $\mathcal G$ that intersects $B(o,\|x\|^3)$} \\ \text{ has diameter (in $\ell^\infty$ norm) smaller than $t$} \end{array}\right\}.
		\end{align*}
	We claim that if~$E_1$,~$E_2$,~$E_3$ all occur, then any path in~$\mathcal G^t$ from~$q(o)$ to~$q(x)$ that minimizes the~$T^t$-passage time does not traverse any extra edge, so that
    \begin{equation} \label{eq:three.events.minimize.path}
        E_1 \cap E_2 \cap E_3 \subseteq \{T^t(q(o),q(x)) = T(x)\}.
    \end{equation}
	Let us prove this. Assume that the three events occur, and fix a path ${\gamma^*=(x_0,\ldots,x_m)}$ in~$\mathcal G^t$ from~$q(o)$ to~$q(x)$ that minimizes the~$T^t$-passage time.   
    Assume for a contradiction that~$\gamma^*$ traverses some extra edge, and let
	\[\mathcal I:= \min\{i: x_{i+1} \notin \mathcal H\},\quad \mathcal J:= \min\{j > \mathcal I: x_j \in \mathcal H\}.\]
	Note that these are well defined with~$0 \le \mathcal I < \mathcal J \le m$, since~$x_0 = q(o) \in \mathcal H$ and~$x_m = q(x) \in \mathcal H$. Define the sub-path~$\gamma := (x_{\mathcal{I}}, x_{\mathcal{I}+1},\ldots, x_{\mathcal J})$. Note that the entirety of~$\gamma^*$, and in particular the two points~$x_{\mathcal I}$ and~$x_{\mathcal J}$, belong to~$B(o,\|x\|^3)$, since~$E_1$ occurs. This will allow us to invoke properties in the definition of~$E_2$ and~$E_3$ when dealing with these points.
    
    Now, there are two cases. First, if~$\|x_{\mathcal J} - x_{\mathcal I}\|\le t$, then (since~$\gamma$ traverses at least one extra edge) we have
	\[T^t(x_{\mathcal I},x_{\mathcal J}) = \sum_{e \in \gamma} \tau^t_e \ge \K t > \thickbar{\beta} t = \thickbar{\beta} \cdot \max\{\|x_{\mathcal I} - x_{\mathcal J}\|,t\},\]
	contradicting the assumption that~$E_2$ occurs. Second, if~$\|x_{\mathcal I} - x_{\mathcal J}\| > t$, then we apply Lemma~\ref{lem_hops} with~$b \le t$ due to $E_3$ to obtain
	\[T^t(x_{\mathcal I},x_{\mathcal J}) = \sum_{e \in \gamma} \tau^t_e \ge \frac{\K t}{3t}\cdot \|x_{\mathcal I} - x_{\mathcal J}\|_\infty > \thickbar{\beta}\|x_{\mathcal{I}} - x_{\mathcal J}\|,\]
	again contradicting the occurrence of~$E_2$. This completes the proof of \eqref{eq:three.events.minimize.path}.

	It remains to show that the three events occur with high probability. First, by Lemma~\ref{eq_new_hopcount},
	\[\p(E_1) > 1 - e^{-ct}.\]
	Second, by~\cref{lm:T.bds.sup}, 
	\[\p(E_2) > 1 - C\|x\|^{4d} \cdot e^{-C't}\]
	if~$\|x\|$ is large enough. Third, by~\cref{lm:PPP.clusters} and Mecke's formula,
	\[\p(E_3) > 1 - (2\|x\|^3)^d\cdot e^{-ct} > 1 - \|x\|^{4d}\cdot e^{-ct}\]
	when~$\|x\|$ is large enough.
\end{proof}

\begin{proof}[Proof of Lemma~\ref{eq_new_hopcount}]
	Recall that~$\gamma^t_{u \leftrightarrow v}$ is the shortest path from~$u$ to~$v$ in~$\mathcal G^t$ that only uses extra edges. We bound
	\[T^t(q(o),q(x)) \le \K t |\gamma^t_{q(o) \leftrightarrow q(x)}| \stackrel{\eqref{eq_hopcount_extra}}{\le} \K t \left(\frac{\sqrt{d}}{t}\|q(o) - q(x)\| +d	\right).\]
	Define the event
	\begin{align*}
		E_1 &:= \{\|q(o)\| \le t/2,\; \|q(x) - x\| \le t/2\}.
	\end{align*}
	 On~$E_1$, we have~$\|q(o)-q(x)\| \le \|q(o)\| + \|x\| + \|q(x) - x\| \le  \|x\| + t$, so
	\begin{align}\label{eq_for_minimizer}
		T^t(q(o),q(x)) \le \K t \left( \frac{\sqrt{d}}{t}(\|x\| + t) + d\right) \le 3 \K d \|x\|,
	\end{align}
	where we used~$\sqrt{d} \le d$ and~$t \le \|x\|$.

	Next, let
	\begin{align*}
		E_2 := \left\{ \begin{array}{l}\text{there is no path $\gamma$ in $\mathcal G^t$ starting in $B(o,t)$}\\[.2cm] \text{with $|\gamma| = \Big\lceil 3\K d \|x\| / \delta\Big\rceil$ and $\sum_{e \in \gamma} \tau^t_e \le 3\K d \|x\|$} \end{array} \right\},
	\end{align*}
	where~$\delta$ is the constant of Lemma~\ref{lem_cheap_hop}. We observe that, on~$E_2$, any path~$\gamma$ in~$\mathcal G^t$ started in~$B(o,t)$ and satisfying~$\sum_{e \in \gamma} \tau^t_e\le  3 \K d \|x\|$ must then also satisfy~$|\gamma| < \lceil 3\K d \|x\|/\delta \rceil$.

	Assume that~$E_1 \cap E_2$ occurs and let~$\gamma^*$ be a path in~$\mathcal G^t$ from~$q(o) = x_0$ to~$q(x) = x_m$ minimizing the~$T^t$-passage time, that is, such that~$T^t(q(o),q(x)) = \sum_{e \in \gamma^*} \tau^t_e$. Then,
	\[\sum_{e \in \gamma^*} \tau^t_e \le \K t |\gamma_{q(o) \leftrightarrow q(x)}^t| \stackrel{\eqref{eq_for_minimizer}}{\le} 3 \K d \|x\|,\]
	so by the observation following the definition of~$E_2$, we have that ${|\gamma^*| < \lceil 3\K d \|x\| / \delta\rceil}$. Then, writing~$\gamma^* = (x_0,\ldots,x_m)$ (so that ${x_0 = q(o)}$, ${x_m = q(x)}$ and ${m = |\gamma^*|}$), for any~$i \in \{0,\ldots, m\}$ we have
	\begin{align*}
		\|x_i\| &\le \|x_0\| + \|x_i - x_0\| \\[.2cm]
		&\le \frac{t}{2} + \sum_{j=0}^{i-1} \|x_{j+1} - x_j\| \\[.2cm]
		&\le \frac{t}{2} + i\cdot \max(r,t) \le \frac{t}{2} + m \cdot  \max(r,t) \le \frac{t}{2} + \lceil 3\K d \|x\| / \delta\rceil \cdot \max(r,t).
	\end{align*}
	Using~$t \le \|x\|$, the right-hand side is smaller than~$\|x\|^3$ if~$\|x\|$ is large enough. We have thus proved that on~$E_1 \cap E_2$,~$\gamma^*$ is entirely contained in~$B(o,\|x\|^3)$.

	To conclude, we note that, by~\cref{prop:Hn.growth},
	\begin{align*}
		\p(E_1) = \p( \mathcal H \cap B_{t/2}(o) \neq \varnothing,\; \mathcal H \cap B_{t/2}(x) \neq \varnothing) > 1 - e^{-c t}
	\end{align*}
	for some~$c > 0$ and~$t$ large enough. Furthermore, by Lemma~\ref{lem_cheap_hop},
	\begin{align*}
		\p(E_2) > 1 - \frac{t^d}{2^{\lceil 3\K d \|x\| / \delta\rceil}} > 1 -e^{-c\|x\|}
	\end{align*}
	for some~$c > 0$ when~$\|x\|$ is large enough.
\end{proof}

\begin{proof}[Proof of Lemma~\ref{lem_hops}]
Let~$\mathcal{X}$ be the set containing $x_0$, $x_n$, and all points of $t \mathbb Z^d$ among $\{x_1,\ldots,x_{n-1}\}$. Take indices~$0=i_0 < i_1 < \cdots < i_m = n$ such that~$\mathcal X = \{x_{i_0},x_{i_1},\ldots, x_{i_m}\}$. The number of extra edges traversed by~$\gamma$ is at least~$m$, because for each~$k =1,\ldots,m$, the jump to~$x_{i_k}$ from the vertex immediately preceding it in~$\gamma$ must use an extra edge. Due to this observation, we bound
	\begin{align}\label{eq_first_nice1}
	\sum_{e \in \gamma} \tau^t_e \ge \K t \cdot m.
\end{align}
	Moreover, for any~$j \in \{0,\ldots,m-1\}$, we have~$\|x_{i_{j+1}}-x_{i_j}\|_\infty \le 2t+b$. This is because the portion of~$\gamma$ from~$x_{i_j}$ to~$x_{i_{j+1}}$ stays inside a single finite cluster of~$\mathcal G$, apart from possibly traversing an extra edge when departing from~$x_{i_j}$, and another extra edge when arriving at~$x_{i_{j+1}}$. Hence,
	\begin{align}\label{eq_second_nice1}
	\|x_n - x_0\|_\infty \le \sum_{j=0}^{m-1} \|x_{i_{j+1}} - x_{i_j}\|_1 \le (2t+b)m.
\end{align}
	Combining~\eqref{eq_first_nice1} and~\eqref{eq_second_nice1}, we obtain the desired inequality.
\end{proof}

%We will now proceed to the proof of the lemma employed to establish the approximation bounds:

%%%%%%%%%%%%--BIBLIOGRAPHY--%%%%%%%%%%%%%%%
\bibliographystyle{abbrvnat}
\bibliography{references}

\end{document}